\documentclass[11pt,leqno]{amsart} 
\usepackage[a4paper,left=24mm,right=24mm,top=32mm,bottom=30mm]{geometry} 
\usepackage{epic,eepic,verbatim,amsmath,amssymb,enumitem} 
\usepackage{graphicx} 
\usepackage{color} 
\usepackage{hyperref} 
\usepackage{subfigure}
\usepackage{esint}

\newcommand{\eps}{\varepsilon}             
\newcommand{\pd}{\partial}                 
                 
\renewcommand{\div}{{\rm{div}}}

\newcommand{\Ha}{\mathcal{H}} 
\newcommand{\dHa}{\,d\Ha^2}

\newcommand{\D}{\nabla}               
\newcommand{\R}{\mathbb{R}}

\newcommand{\F}{\mathcal F}

\newcommand{\SL}{\Delta_{\Gamma}}          
  
\newcommand{\SG}{\nabla_{\Gamma}}

\newcommand{\normg}[1]{ \left\|#1\right\|_{\Gamma}}

\newcommand{\MW}{-\!\!\!\!\!\!\!\!\;\int}

\newcommand{\jump}[1]{\left[ #1 \right]_-^+}
\newcommand{\GammaT}{\boldsymbol{\Gamma}}

\newcommand{\spa}{\operatorname{span}}

\renewenvironment{proof}[0] {\noindent{\em Proof.}}{\hfill \qed\\[1ex] }

\theoremstyle{plain}
\numberwithin{equation}{section}
\newtheorem{lemma}{Lemma}[section]

\newtheorem{proposition}[lemma]{Proposition}

\theoremstyle{definition}
\newtheorem{remark}[lemma]{Remark}

\newtheorem{siproblem}[lemma]{Sharp Interface Problem}

\def\cprime{$'$}

\begin{document} 

\title[Surface Cahn-Hilliard 
modeling  raft formation
]
{A coupled surface-Cahn--Hilliard bulk-diffusion system
modeling lipid raft formation
in cell membranes}

\author[H. Garcke]{Harald Garcke}

\author[J. Kampmann]{Johannes Kampmann}

\author[A. R\"atz]{Andreas R\"atz}

\author[M. R\"oger]{Matthias R\"oger}



\begin{abstract}
We propose and investigate a model for lipid raft formation and dynamics in biological membranes. The model describes the lipid composition of the membrane and an interaction with cholesterol. To account for cholesterol exchange between cytosol and cell membrane we couple a bulk-diffusion to an evolution equation on the membrane. The latter describes a relaxation dynamics for an energy taking lipid-phase separation and lipid-cholesterol interaction energy into account. It takes the form of an (extended) Cahn--Hilliard equation. Different laws for the exchange term represent equilibrium and non-equilibrium models. We present a thermodynamic justification, analyze the respective qualitative behavior and derive asymptotic reductions of the model. In particular we present a formal asymptotic expansion near the sharp interface limit, where the membrane is separated into two pure phases of saturated and unsaturated lipids, respectively. Finally we perform numerical simulations and investigate the long-time behavior of the model and its parameter dependence. Both the mathematical analysis and the numerical simulations show the emergence of raft-like structures in the non-equilibrium case whereas in the equilibrium case only macrodomains survive in the long-time evolution.
\end{abstract}
\date{\today}
\maketitle
%=======================
%
%=======================
\section{Introduction}
Phase separation processes that lead to microdomains of a well-defined length-scale below the system size arise in various physical and biological systems. A prominent example is the microphase separation in block-copolymers \cite{BaFr99} or other soft materials, characterized by a fluid-like disorder on molecular scales and high degree of order at larger scales. Here micro-scale pattern arise by a competition between thermodynamic forces that drive (macro-)phase separation and entropic forces 
that limit phase separation. Different mathematical models for microphase separation in di-block copolymers have been developed, in particular built on self-consistent mean field theory (see for example \cite{MaSc94} and the references therein) or density functional theory developed in \cite{Leib80,OhKa86,BaOo90} that leads to the so-called Ohta--Kawasaki free energy. 
Diblock-copolymer type models have also been studied on spheres and more
general curved surfaces \cite{LiZQ14,VKCF11,CBHC07,LFZQ06} and show a
variety of different stripe- or spot-like patterns. In contrast to
materials science microphase separation on biological membranes is much less understood, in particular in living cells. In this contribution we present and analyze a  model for so-called lipid rafts, that represent microdomains of specific lipid compositions.

The outer (plasma) membrane of a biological cell consists of a bilayer formed by several sorts of lipid molecules and contains various other molecules like proteins and cholesterol. Besides being the physical boundary of the cell the plasma membrane also plays an active part in the functioning of the cell. Many studies over the past decades have shown that the structure of the outer membrane is heterogeneous with several microdomains of different lipid and/or protein composition. Such domains have been clearly observed in artificial membranes such as giant unilamellar vesicles \cite{VeKe03,VeKe05}. Here a `less fluidic' (liquid-ordered) phase of saturated lipids and cholesterol separates from a `more fluidic' (liquid-disordered) phase of unsaturated lipids. The nucleation of dispersed microdomains is followed by a classical coarsening process that leads to coexisting domains with a length scale of the order of the system size. The situation is much more complex and much less understood for living cells. 
Their plasma membrane represents a heterogeneous structure with a complex and dynamic lipidic organization. The formation, maintenance and dynamics of intermediate-sized domains (10 - 200 nm on cells of $\mu m$ size \cite{LN}) called `lipid rafts' are of prime interest. These are characterized as liquid-ordered phases that consist of  saturated lipids and are enriched of cholesterol and various proteins \cite{SiIk97,BrLo98}. These rafts contribute to various biological processes including signal transduction, membrane trafficking, and protein sorting \cite{FaSH10}. It is therefore an interesting question to study the underlying process, by which these rafts are generated and maintained, and to understand the mechanism that allows for a dynamic distribution of intermediate-sized domains.

Several phenomenological mesoscale models have been proposed, see for example the review \cite{FaSH10}. One class of models argues that raft formation is a result of a thermodynamic equilibrium process. Here one contribution is a phase separation energy that would induce a reduction of interfacial size between the raft and non-raft phases. The observation of nano-scale structures is then explained by including different additional energy contributions. One proposal is that thermal fluctuations near the critical temperature for the phase separation are responsible for the intermediate-sized structures. Other explanations consider interactions between lipids and membrane proteins that could act as a kind of surfactant or could `pin' the interfaces due to their immobility \cite{YeWe07,SCGS06,WiBaVo12}. Finally, raft formation could also be stabilized by induced changes of the membrane geometry and elasticity effects. However, as argued in \cite{FaSH10}, all such models are not able to reproduce key characteristics of 
raft dynamics; non-equilibrium effects essentially contribute to raft
formation. Of particular importance are active transport processes of
raft components that allow to maintain a non-equilibrium composition of
the membrane. The competition between phase separation and recycling of
raft components is argued to be of major importance for the dynamics and structure of lipid rafts.

Foret \cite{Fore05} proposed a simple mechanism of raft formation in a
two-component fluid membrane. This model includes a constant exchange of
lipids between the membrane and a lipid reservoir as well as a typical
phase separation energy. While relaxation of the latter tends to create
large domains, the constant insertion and extraction of lipids in the
membrane ensures indeed the formation of rafts. The emerging
microdomains are static (in contrast to the lipid rafts on actual cell
membranes) and the size distribution of these rafts is rather uniform,
whereas in vivo cell membranes show a dynamic distribution of rafts of
different sizes (see the concluding remarks in \cite{Fore05} and the
discussion of Foret’s model in \cite{FaSH10}). Moreover, as already
noticed in \cite{FaSH10}, whereas Foret’s model is motivated by including
non-equilibrium effects his model can be equivalently characterized as relaxation dynamics for an effective energy given by the Ohta--Kawasaki energy of block-copolymers \cite{
OhKa86} that is well-known to generate phase-separation in intermediate-sized structures.

A similar model by G{\'o}mez, Sagu{\'e}s and Reigada \cite{GoSR08} considers a ternary mixture of saturated and unsaturated lipids together with cholesterol and studies the interplay of lipid phase separation and a continuous recycling of cholesterol. An energy that is determined by the relative concentration $\phi$ of saturated lipids and the relative cholesterol concentration $c$ is proposed that in particular includes a phase separation energy of Ginzburg--Landau type for the lipid phases and a preference for cholesterol--saturated lipid interactions over cholesterol--unsaturated lipid interactions. The dynamics include an exchange term for cholesterol that is given by an in-/out-flux proportional to the difference from a constant equilibrium concentration of cholesterol at the membrane.

Our aim here is to propose an extended model and to present both a
mathematical analysis and numerical simulations. Similar as in
\cite{GoSR08} one ingredient of our model are energetic contributions
from lipid phase separation and lipid-cholesterol interaction. In
addition, we include the dynamics of cholesterol inside the cytosol (the liquid matter inside the cell) and prescribe a detailed coupling between the processes in the cell and on the cell membrane. In particular, the outflow of cholesterol from the cytosol appears as a source-term in the membrane-cholesterol dynamic and will be characterized by a constitutive relation. We will investigate different choices of this relation and will illustrate the implications on the emergence of microdomains.

\subsection{A lipid raft model including cholesterol exchange and cytosolic diffusion}
To give a detailed description of our model let us fix an open bounded set $B\subset\R^3$ with smooth boundary $\Gamma=\partial B$ representing the cell volume and cell membrane, respectively. Let $\varphi$ denote a rescaled relative concentration of saturated lipid molecules on the membrane, with $\varphi=1$ and $\varphi=-1$ representing the pure saturated-lipid and pure unsaturated-lipid phases, respectively. Moreover let $v$ denote the relative concentration of membrane-bound cholesterol, where $v=1$ indicates maximal saturation, and let $u$ denote the relative concentration of cytosolic cholesterol. We then prescribe a phase-separation and interaction energy of the form
\begin{equation}
  \label{eq:energy}
  \F(v, \varphi) = \int_\Gamma\Big( \frac\eps2|\nabla_\Gamma\varphi|^2 +
  \eps^{-1}W(\varphi) + \frac1{2\delta}(2v - 1 -
\varphi)^2\Big)d{\mathcal H}^2 
\end{equation}
with $W$ a double-well potential that we choose as $W(\varphi)=
\frac{1}{4}(1-\varphi^2)^2$, constants $\eps,\delta>0$, and
$\nabla_\Gamma$ denoting the surface gradient. The first two terms
represent a classical Ginzburg-Landau phase separation energy, whereas
the third term models a preferential binding of cholesterol to the lipid-saturated phase. We define the chemical potentials
\begin{align}
  \mu &:= \frac{\delta \F}{\delta \varphi} 
  =  - \eps \SL \varphi + \eps^{-1}W'(\varphi) - \delta^{-1}(2v - 1 - \varphi),\label{eq:orig_mu}\\ 
\theta &:= \frac{\delta \F}{\delta v} = \frac2\delta(2v - 1 - \varphi)\label{eq:orig_theta},
\end{align}
where $\SL$ denotes the Laplace--Beltrami operator on $\Gamma$, and prescribe for the dynamics of the concentrations over a time interval of observation $(0,T)$ the following system of equations 
\begin{alignat}{2}
  \label{eq:diffU}
  \partial_t u &= D \Delta u &\qquad \text{in } B \times (0,T]&,\\
  \label{eq:flux}
  - D \nabla u \cdot \nu & = q 
  \quad&\text{on } \Gamma \times (0,T]&,\\
  \label{eq:CH1}
  \pd_t \varphi &= \SL \mu
  &\text{on } \Gamma \times (0,T]&,\\
  \label{eq:v}
  \pd_t v &= \SL \theta + q = \frac4\delta\SL v -  \frac2\delta\SL
  \varphi + q
  &\text{on } \Gamma \times (0,T],&
\end{alignat}
where $\nu$ denotes the outer unit-normal field of $B$ on $\Gamma$. The system is complemented with initial conditions $u_0,\varphi_0,v_0$ for $u$, $\varphi$ and $v$, respectively. The first equation
represents a simple diffusion equation for the cholesterol in the
bulk, the second equation characterizes the outflow of cholesterol. The third equation is a Cahn--Hilliard dynamics for the lipid concentration on the membrane, whereas the equation for the cholesterol on the membrane combines 
a mass-preserving relaxation of the interaction energy and an exchange with the bulk reservoir of cholesterol given by the flux from the cytosol. This combination yields
a diffusion equation with cross-diffusion contributions and a source term.

To close the system it remains to characterize the exchange term $q$. We follow here two possibilities: First we prescribe a constitutive relation by considering the membrane attachment as an elementary `reaction' between free sites on the membrane and cholesterol, and the detachment as proportional to the membrane cholesterol concentration, expressed by the choice
\begin{align}
	q \,=\, c_1 u (1 - v) - c_2 v. \label{eq:q-1}
\end{align}
A similar coupling of bulk--surface equations has been investigated in a
reaction-diffusion model for signaling networks \cite{RR,RaRo14}.  As a
second possibility we consider choices of $q$ that allow for a global
free energy inequality for the coupled membrane/cytosol system. These two different cases could be considered as a distinction between 
open and closed systems and one important aspect of this work is to evaluate the consequences of these choices for the formation of complex phases.

We remark  that the system conserves both the total cholesterol and the lipid concentrations since for arbitrary choices of $q$ we obtain the relations
\begin{align}
  \frac{d}{dt} \left(\int_B u \,dx+ \int_\Gamma v\, d{\mathcal H}^2 \right) = 0,\\
  \frac{d}{dt}\int_\Gamma \varphi \,d{\mathcal H}^2  = 0.
\end{align}

Let us contrast the above model with the Ohta--Kawasaki model for phase separation in diblock copolymers mentioned above. Let $\Omega\subset \R^n$ denote a spatial domain, $\varphi$ the relative concentration of one of the two polymers and let $m:=\fint\varphi$ be the prescribed average of $\varphi$ over $\Omega$. The mean field potential $z$ is then given by
\begin{align*}
	-\Delta z  \,=\, \varphi-m \quad\text{ in }\Omega,\qquad \nabla z\cdot\nu_\Omega\,=\, 0\quad\text{ on }\partial\Omega,\qquad \int_\Omega z \,=\, 0.
\end{align*}
Then a free energy is prescribed of the form
\begin{align}
	\F_{OK}(\varphi) \,=\, \int_\Omega
\Big(\frac\eps2|\nabla\varphi|^2 +
  \eps^{-1}W(\varphi)\Big) \,dx+ \frac{\sigma}{2}\int_\Omega |\nabla
z|^2\,dx , \label{eq:OK-intro}
\end{align}
where $\sigma>0$ is a fixed constant. Note that the last term can also be written as $\frac{\sigma}{2}\|\varphi-\fint\varphi\|_{H^{-1}}^2$. A relaxation dynamics of Cahn--Hilliard type is then typically considered given 
as
\begin{align}
	\partial_t\varphi \,=\, \Delta\mu,\qquad \mu\,=\, \frac{\delta \F_{OK}}{\delta \varphi} 
  \,=\,  - \eps \Delta \varphi + \eps^{-1}W'(\varphi)+ \sigma z. \label{eq:OK-dyn}
\end{align}
The energies $\F$ and $\F_{OK}$ both contain a Cahn--Hilliard type energy contribution that favors macro-phase separation but are different in the additional terms. However we will see below that stationary patterns for our lipid raft model with the choice \eqref{eq:q-1} are for small $\delta>0$ closely related to stationary points of $\F_{OK}$. If on the other hand one considers simple choices for $q$ that lead to a 
global free energy inequality, we will observe a macro-scale separation of all saturated lipids in one connected domain. This indicates that in fact non-equilibrium processes are responsible for lipid raft formation.

Several asymptotic regimes are interesting in view of the different
parameters included in our model. We will investigate the limit
$\eps\to 0$ that corresponds to a strong segregation limit and leads
to a model where no mixing of the lipid phases is allowed and the
domain splits into regions where $\varphi=1$ and $\varphi=-1$
respectively. This limit corresponds to the sharp interface limit in
phase-field models and connects to the analysis of the Ohta--Kawasaki
model in \cite{NiOh95}. Since the cytosolic diffusion in biological
cells is known to be much faster than lateral diffusion on the cell
membrane another natural reduction of the model appears in the limit
$D\to\infty$ that leads to a non-local model defined solely on the
cell membrane. Finally, assuming the effect of the lipid
interaction with cholesterol to be large motivates to consider the
asymptotic regime $\delta\downarrow 0$.

\subsection{Outline of the paper and main results}
In the next section  we will derive the model \eqref{eq:diffU}-\eqref{eq:v} from thermodynamic considerations. In particular we will show that for arbitrary choices of the exchange term $q$ the surface equations and the bulk (cytosolic) equations are thermodynamically consistent when viewed as separate systems. Depending on the specific choice of $q$ we may or may not have a global (that is with respect to the full model) free energy inequality. We will present examples for both cases. \\
In Section \ref{sec:qualitative} we will first derive a reduced raft
model in the large cytosolic diffusion limit by formally taking
$D\to\infty$. We then analyze the qualitative behavior of the
(reduced) system in terms of a characterization of stationary points
and an investigation of their relation to stationary points of the
Ohta--Kawasaki model. A formal asymptotic expansion for the sharp
interface reduction $\eps\to 0$ of our raft model is presented in
Section \ref{sec:sharpinterface}. Here we also briefly discuss  the
resulting limit problem that takes the form of a free boundary problem
of Mullins--Sekerka type on the membrane with an additional coupling
to a diffusion process in the bulk and including an interaction with
the cholesterol concentration. In Section \ref{sec:simulations} we
present numerical simulations of the full and reduced raft model. In
particular, we study spinodal decomposition, coarsening scenarios
and the possible appearance of raft-like structures as (almost) stationary
states for different choices of the exchange term $q$ and for different parameter regimes. Some conclusions are stated in the final section.

%==========================================
% thermodynamics
%==========================================
\section{Thermodynamic justification of the lipid raft model}\label{sec:thermodynamics}
In this section we will derive the governing equations for the lipid raft model
from basic thermodynamical conservation laws using a free energy inequality.
Our arguments are similar to an approach used by Gurtin
\cite{Gurtin1989, Gurtin1996}, who derived the Cahn-Hilliard equation in the context
of non-equilibrium thermodynamics. We first of all consider the equations which
have to hold on the surface, will then consider the equations in the bulk and
subsequently
we will couple both systems. 

The basic quantities on the membrane surface are the
 rescaled relative concentration of the saturated lipid molecules $\varphi$, the
concentration $v$ of the membrane-bound cholesterol, the mass flux $J_\varphi$
of the lipid molecules, the mass flux $J_v$ of the cholesterol, the mass supply of
cholesterol $q$, the surface free energy density $f$ and the chemical potential
$\mu$ related to the lipid molecules and the chemical potential $\theta$ related to
the surface cholesterol. The underlying laws for any surface subdomain
$\Sigma\subset \Gamma$ are the mass balance for the lipids
\begin{equation}
\label{MB1}
\frac d{dt} \int_\Sigma  \varphi\,  d{\mathcal H}^2  = -\int_{\partial \Sigma}  
\,J_\varphi \cdot n \, d{\mathcal H}^1,
\end{equation}
the mass balance for the surface cholesterol
\begin{equation}
\label{MB2}
\frac d{dt} \int_\Sigma  v \, d{\mathcal H}^2  = -\int_{\partial \Sigma}
\,J_v \cdot n \, d{\mathcal H}^1 + \int_\Sigma \,q\, d{\mathcal H}^2
\end{equation}
and the second law of thermodynamics, which in the isothermal situation has the form
\begin{equation}
\label{MB3} 
\frac d{dt} \int_\Sigma  f d{\mathcal H}^2  \leq  -\int_{\partial \Sigma}
\,(\mu J_\varphi \cdot n -
(\partial_t \varphi f_{,\nabla_\Gamma \varphi} \cdot n) +
\theta J_v \cdot n)  \, d{\mathcal H}^1 + \int_\Sigma \,\theta q\, d{\mathcal H}^2
\end{equation}
where $f$ denotes the surface free energy density and where $n$ is the outer unit conormal to $\partial \Sigma$ in the tangent space
of $\Gamma$. In addition, we denote by $d{\mathcal H}^d$ the integration with respect
to the $d-$dimensional surface measure and $f_{,\nabla_\Gamma \varphi}$
denotes the partial derivatives of $f$ with respect to
the variables related to $\nabla_\Gamma\varphi$ in a constitutive
relation
$f=f(..., \nabla_\Gamma \varphi,...)$. Similarly we will denote with a
subscript comma partial derivatives with respect to other variables.
 For a discussion of these
laws in cases in which source terms are present and at the same time 
 $f$ does not
depend on $\nabla_\Gamma\varphi$  we refer to
\cite{Gurtin1989}, \cite[Chapter 62]{GFA},
and \cite{PodioG2006}.  
Thermodynamical models of phase transitions with an order parameter
typically involve a free energy density $f$ which does depend on
$\nabla_\Gamma \varphi$.
In this case  the free energy flux does not only  involve
the classical terms $\mu J_\varphi$  and $\theta J_v$ but also
a term $ \partial_t  \varphi f_{,\nabla_\Gamma \varphi}$. This is discussed
in \cite{Gurtin1996,AltP1994,AGG}. Gurtin \cite{Gurtin1996}
introduces a microforce balance involving a microstress $\bf \xi$ in
order to derive the Cahn-Hilliard equation. Here, we do not discuss the microforce
balance and instead already use the form ${\bf \xi} =
 f_{,\nabla_\Gamma \varphi}$ which could be  derived in our context in the same way as
in \cite{Gurtin1996}. However, in order to shorten the
presentation we do not state the details.
We hence obtain (\ref{MB3}) as the relevant free energy
inequality in cases where no external microforces are present.

Since the above (in-)equalities (\ref{MB1})--(\ref{MB3}) hold for all
$\Sigma$, we
obtain with the help of the Gau{\ss} theorem on surfaces the local forms,
compare \cite{Gurtin1996},
\begin{alignat}{2}
\label{CL1}
\partial_t \varphi  +\div_\Gamma J_\varphi  &=0
&\qquad \text{in } \Gamma \times (0,T]&,\\
\label{CL2}
\partial_t v +\div_\Gamma J_v  &=q &\qquad \text{in } \Gamma \times
(0,T]&,
\\
\partial_t f+ \div_\Gamma( \mu J_\varphi 
-  \partial_t \varphi
f_{,\nabla_\Gamma \varphi}  + \theta J_v )&\leq \theta q
&\qquad \text{in } \Gamma \times (0,T]&.
\end{alignat}
With the constitutive relation $f=f(v, \varphi, \nabla_\Gamma \varphi)$
we obtain from the local form of the  free energy inequality
\begin{align*}
f_{,\varphi} \partial_t \varphi 
+f_{,v} \partial_t v +& \nabla_\Gamma \mu \cdot J_\varphi + \nabla_\Gamma \theta \cdot J_v +
\\
&(\div_\Gamma J_\varphi) \mu + (\div_\Gamma J_v) \theta - \partial_t\,
\varphi \div_\Gamma   f_{,\nabla_\Gamma \varphi} \leq 
\theta q.
\end{align*}
Using the  conservation laws  (\ref{CL1}) and (\ref{CL2}) we obtain
$$
(f_{,\varphi} - \div_\Gamma (f_{,\nabla_\Gamma \varphi} ) -\mu)\partial_t \varphi +
( f_{,v} -\theta)\partial_t v + \nabla_\Gamma \mu\cdot J_\varphi + \nabla_\Gamma \theta \cdot J_{v} \leq 0.
$$
The fact that solutions of the conservation laws
with arbitrary values for $\partial_t \varphi $ and $\partial_t v$ 
can appear is used in the theory of rational thermodynamics
to show that the factors multiplying 
$\partial_t \varphi $ and $\partial_t v$ have to disappear as they do
not depend on $\partial_t \varphi $ and $\partial_t v$.
We refer to Liu's method of Lagrange multipliers
\cite{Liu} and to \cite{GFA,AltP1994} for a more precise discussion on
how the free energy inequality can be used to restrict possible
constitutive relations. 
We now choose the following  constitutive relations which guarantee
that the free energy inequality is fulfilled for arbitrary solutions of
(\ref{CL1}), (\ref{CL2}). In fact, choosing
\begin{align}
\mu&=f_{,\varphi} - \div_\Gamma (f_{,\nabla_\Gamma \varphi} ), \\
\theta&= f_{,v} ,\\
J_\varphi &= - D_\varphi \nabla_\Gamma  \mu, \\
J_v&= - D_v \nabla_\Gamma \theta 
\end{align}
with $D_\varphi, D_v \geq 0$ ensures that the free energy inequality is fulfilled for all
solutions of the conservation laws.
More general models, e.g. taking cross diffusion into account, are possible and we refer to
\cite{AltP1994} for an approach which can be used to obtain more general models.

We now consider the governing physical laws in the bulk. As variables we
choose $u$ which is the relative concentration of the cytosolic cholesterol, the
bulk chemical potential $\mu_u$, the
bulk free energy density $f_b(u)$, the bulk flux $J_u$ 
and the surface mass source term $q_u$. We need to fulfill
the following mass balance equation in integral form
which has to hold for all open $U\subset B$:
\begin{equation}
\label{IMB}
\frac d{dt} \int_U \,u\,dx = -\int_{(\partial U)\setminus \Gamma}
\,J_u\cdot \nu\, 
d{\mathcal H}^2 + \int_{(\partial U)\cap \Gamma} q_u  d{\mathcal H}^2
\end{equation}
and the free energy inequality
\begin{equation}
\label{FEB}
\frac d{dt} \int_U \,f_b(u)\,dx \leq -\int_{(\partial U) \setminus \Gamma}
\,\mu_u J_u\cdot \nu\, 
d{\mathcal H}^2 + \int_{(\partial U)\cap \Gamma} \mu_u q_u  d{\mathcal H}^2
\end{equation}
which also has to hold for all open $U \subset B$.
Here we allowed for a source term $q_u$ on $\Gamma$ and in accordance to our discussion above we introduced the free energy source $\mu_u q_u$
in $(\ref{FEB})$. As on the interface the free energy source term is given classically as a product of the mass source term and the chemical potential.
If we use the fact that we can choose an arbitrary open set $U$ which is
compactly supported in $B$  (which then implies $(\partial U) \cap
\Gamma =\emptyset$)
we obtain with the help of the Gau{\ss} theorem
\begin{alignat}{2}
\label{BMB} \partial_t u + \div J_u &=0 &\qquad \text{in }B 
\times (0,T]&, \\
\label{PFEB} \partial_t f_b(u) + \div (\mu_u J_u)&\leq 0
&\qquad \text{in }B
\times (0,T]&.
\end{alignat} 
Choosing
\begin{align}
\label{eq:mu_u}
\mu_u&= f_b^\prime (u),\\
J_u&=-M(u) \nabla (f^\prime_b (u))
\end{align}
makes sure that (\ref{PFEB}) is true for all solutions of (\ref{BMB}). 
In the following we will often choose $M(u)
=\frac{D_u}{f_b^{\prime\prime}(u)}$ and this will lead to the linear diffusion equation \eqref{eq:diffU}
$$\partial_t u - D_u \Delta u=0 \qquad \text{in }B
\times (0,T].$$
Choosing  $U$ such that $\partial U\cap \Gamma\neq \emptyset$ we obtain from (\ref{IMB})
$$0=\int_U\,  (\partial_t u + \div J_u) \,dx = \int_{\partial
U\cap\Gamma}
(q_u +J_u\cdot \nu) d{\mathcal H}^2$$
which gives, since $U$ is arbitrary,
$$q_u =-J_u \cdot \nu \;(= D_u \nabla u\cdot \nu ) \qquad \text{in
}\Gamma
\times (0,T], $$
and which yields \eqref{eq:flux} for $q=-q_u$. We will now state global balance laws in the case where the mass supply for the interface stems from the bulk and vice versa.
\begin{lemma}\label{l:free_energy_ineq}
We assume that the above stated mass balance equations hold for the bulk
and the surface and assume in addition that $q_u=-q$. Then it holds
$$\frac d{dt} \Big( 
 \int_B \,u\,dx + \int_\Gamma \,v\, d{\mathcal H}^2 \Big)=0
\qquad \frac d{dt} \int_\Gamma \,\varphi \, d{\mathcal H}^2=0  .$$
If in addition, the free energy inequalities in the bulk and on the
surface are true it also holds 
$$\frac d{dt} \Big( \int_\Gamma \, f(v,\varphi,\nabla_\Gamma \varphi) \,
d{\mathcal H}^2 + \int_B\, f_b(u)\, dx \Big) \leq \int_\Gamma q(\theta
-\mu_u) d{\mathcal H}^2.$$
\end{lemma}
\begin{proof}
The first two equations  follow from (\ref{MB1}), 
(\ref{MB2}) with $\Sigma=\Gamma$ and
(\ref{IMB}) with $U=B$ since $\partial U=\Gamma$ and 
since $\partial \Gamma
 =\emptyset$.
The total free energy inequality follows similarly from (\ref{MB3}) and (\ref{FEB}).
\end{proof}
\begin{remark}
\setlength{\leftmargini}{1.5em}\begin{enumerate}[label=(\roman*)]
\item In the above lemma we chose $q_u=-q$, that is the mass lost on the surface generates a source of mass for the bulk.

\item Several constitutive laws for $q$ make sense. It is possible to consider
\begin{equation} 
\label{TCQ}
q=-c (\theta -\mu_u), \qquad c\geq 0
\end{equation}
which leads to 
 model for which the total free energy decreases. In this case we obtain 
$$ \frac d{dt} \Big( \int_\Gamma \, f(v,\varphi,\nabla_\Gamma \varphi)
\, d{\mathcal H}^2 + \int_B\, f_b(u)\, dx \Big) \leq - \int_\Gamma \, c
(\theta -\mu_u)^2\, d{\mathcal H}^2 \leq 0.$$
\item We also consider the reaction type source term, compare (\ref{eq:q-1}),
$$q=c_1 u(1-v) + c_2 v.$$
Also in this case we obtain a consistent model which fulfills the bulk
and surface free energy inequalities with source terms as stated above. However, in this
case the total free energy as the sum of the bulk and the surface free energy might increase which can be due to the fact that we neglect energy contributions
generated by the detachment and attachment process.

\item\label{rem:thermodynamic_spec_choices} One possible choice for the surface free energy density is, compare
(\ref{eq:energy}),
$$f(v,\varphi, \nabla_\Gamma \varphi) =\frac{\gamma \eps}2 |\nabla_\Gamma \varphi |^2 +\frac \gamma \eps  W(\varphi) +\frac 1{2\delta} (2v-1-\varphi)^2,\qquad \eps,\delta,\gamma>0.$$
For $f_b(u)=\frac{1}2 u^2$, $q_u=-q$ and $D_u=D$, $\gamma=D_\varphi=D_v=1$ we obtain the system
\eqref{eq:orig_mu}--\eqref{eq:v}. \\
With the above  quadratic choice for $f_b$ and arguing as in the
derivation
of the energy inequality one can prove the identity
\begin{align}
\nonumber
\frac d{dt} \Big( \int_\Gamma  \frac {\eps}2 |\nabla_\Gamma \varphi |^2
+\frac{1}{\eps}  W(\varphi)& +\frac 1{2\delta} (2v-1-\varphi)^2\,
d{\mathcal H}^2 + \int_B\, \frac{1}{2}u^2\, dx \Big)
\\& +\int_\Gamma  ( |\nabla_\Gamma \mu|^2 + |\nabla_\Gamma
\theta|^2 )d{\mathcal H}^2
+\int_B D |\nabla u|^2  = \int_\Gamma q(\theta - u)
\label{EnergyDerivative}
\end{align}
where the last term is non-positive for the choice (\ref{TCQ}). Any other evolution that is based on a choice of $q$ such that the right-hand side of \eqref{EnergyDerivative} is always non-positive decreases the total free energy. This is in particular the case for choices of $q$ such that \eqref{eq:orig_mu}-\eqref{eq:v} can be characterized as a gradient flow.
\end{enumerate}

In the following we are mainly interested in the dependence on the parameters $\eps,\delta,D_u$ and therefore as in \ref{rem:thermodynamic_spec_choices} above we always set $D_u=D$, $\gamma=D_\varphi=D_v=1$, in which case the above choices of free energy densities and mass fluxes yield the system \eqref{eq:orig_mu}--\eqref{eq:v}.
\end{remark}
%==========================================
% qualitative behavior
%==========================================
\section{Qualitative behavior}\label{sec:qualitative}
In this section we will investigate qualitative properties of the model \eqref{eq:orig_mu}-\eqref{eq:v} and of the asymptotic reduction in the large cytosolic diffusion limit that we derive below. One key question here is whether or not our lipid raft model supports the formation of mesoscale patterns. We will distinguish different choices for the exchange term $q$ and compare evolutions that reduce the total free energy with the evolution for the choice $q$ given by the reaction-type law \eqref{eq:q-1}, which we consider as a prototype of a non-equilibrium model. We remark that most of the arguments in this section are purely formal; a rigorous justification is out of the scope of the present paper. In particular, we assume the existence of smooth solutions, their convergence to stationary states as times tends to infinity, and that the long-time behavior of the full system asymptotically agrees with that of the reduced system developed below. 

We first observe that under an additional growth assumption on the exchange
term $q$ we obtain energy bounds, even in the case that the system does
not satisfy a global energy inequality.
\begin{proposition}\label{prop:energy-bound}
Assume that $q$ has at most linear growth, that is there exists $\Lambda>0$
such that
\begin{align}
        |q(\varphi,u,v)| \,\leq\, \Lambda(1+|\varphi|+|u|+|v|)\quad\text{ for all
}\varphi,u,v \in\R. \label{eq:ass-q}
\end{align}
Then for all $0<t<T$ and all $D\geq D_0>0$, $0<\eps\leq \eps_0$ any solution of \eqref{eq:orig_mu}-\eqref{eq:v} with initial data $\varphi_0,u_0,v_0$ satisfies
\begin{align}
        \F(v(\cdot,t),\varphi(\cdot,t)) + \frac{1}{2}\int_B u(\cdot,t)^2
+ \int_0^t\int_B \frac{D}{2}|\nabla u|^2 
	\,\leq\, C(\delta,\Lambda,T,D_0,\eps_0,v_0,\varphi_0,u_0). \label{eq:eb-1}
\end{align}
\end{proposition}
\begin{proof}
From \eqref{EnergyDerivative}  we deduce that
\begin{align}
        &\frac{d}{dt}\Big(\F(v(\cdot,t),\varphi(\cdot,t)) +
\frac{1}{2}\int_B u(\cdot,t)^2\Big) \label{eq:prop3-0}\\
        =\,& -\int_B D|\nabla u|^2(\cdot,t) - \int_\Gamma \left[ |\SG
\mu|^2(\cdot,t) +|\SG\theta|^2(\cdot,t) - (\theta
-u)(\cdot,t)q(\varphi(\cdot,t),u(\cdot,t),v(\cdot,t))\right].\notag
\end{align}
For the last term we use the estimate
\begin{align}
       \Big| \int_\Gamma (\theta-u)q(\varphi,u,v)\Big| &\,\leq\, \int_\Gamma \Big(\theta^2 + u^2 + C_\Lambda (1+\varphi^2 + u^2 +
v^2)\Big) \nonumber \\ &\,=\, \int_\Gamma \theta^2 + C_\Lambda \int_\Gamma (1+\varphi^2) + (C_\Lambda+1) \int_\Gamma u^2 + C_\Lambda \int_\Gamma v^2. \label{eq:prop3-1}
\end{align}
For the second term on the right-hand side we obtain
\begin{align}
	\int_\Gamma (1+\varphi^2) \,\leq\, C\big(1+\int_\Gamma W(\varphi)\big)\,\leq\, C(\eps_0)\big(1 +\F(v(\cdot,t),\varphi(\cdot,t))\big).\label{eq:prop3-2}
\end{align}
Using  \cite[Chapter 2, (2.25)]{LaUr68} the third term on the right-hand side of \eqref{eq:prop3-1} can be estimated by bulk quantities,
\begin{align*}
        \int_\Gamma u^2 \,\leq\, \frac{D}{2}\int_B |\nabla u|^2 + C(D)\int_B u^2 
\end{align*}
for a suitable constant $C(D)$, which in particular yields for all $D\geq D_0$
\begin{align}
        \int_\Gamma u^2 \,\leq\, \frac{D}{2}\int_B |\nabla u|^2 + C(D_0)\int_B u^2. \label{eq:prop3-2a}
\end{align}
To estimate the last integral on the right-hand side of \eqref{eq:prop3-1} we first use Young's inequality and \eqref{eq:orig_theta} to obtain
\begin{align*}
        \frac{16}{\delta^2}v^2 \,=\, \theta^2
+\frac{4}{\delta}\theta(1+\varphi)
+\frac{4}{\delta^2}(1+\varphi)^2\,\leq\, 2\theta^2 + \frac{8}{\delta^2}(1+\varphi)^2,
\end{align*}
and further deduce that
\begin{align}
        v^2  \,\leq\, \frac{\delta^2}{8} \theta^2 + C(1+ W(\varphi)), \notag\\
        \int_\Gamma v^2(\cdot,t) \,\leq\, C(\delta,\eps_0)\Big(1+\F(v(\cdot,t),\varphi(\cdot,t))\big). \label{eq:prop3-3}
\end{align}
We therefore obtain from \eqref{eq:prop3-0}-\eqref{eq:prop3-3} that
\begin{align*}
        &\frac{d}{dt}\Big(\F(v(\cdot,t),\varphi(\cdot,t)) +
\frac{1}{2}\int_B u(\cdot,t)^2\Big)  + \frac{D}{2}\int_B |\nabla
u|^2(\cdot,t)+ \int_\Gamma \Big( |\SG \mu|^2(\cdot,t)
+|\SG\theta|^2(\cdot,t)\Big)\\
        \leq\, &C(\delta,\Lambda,D_0,\eps_0)\Big(\F(v(\cdot,t),\varphi(\cdot,t))
+\frac{1}{2}\int_B u(\cdot,t)^2\Big)+C(\delta,\Lambda).
\end{align*}
By the Gronwall inequality we deduce the claim.
\end{proof}
Note that for the choice \eqref{eq:q-1} of $q$ assumption \eqref{eq:ass-q} is not satisfied. However, for any modification that coincides with that choice on a bounded domain in the $u,v$ plane and that has at most linear growth outside the conclusion of Proposition \ref{prop:energy-bound} holds.  For \eqref{TCQ} or any other choice of $q$ that implies a total free energy inequality we obtain an even better estimate, since now the right-hand side in \eqref{EnergyDerivative} is non-positive.

\subsection{A reduced model in the limit of large cytosolic diffusion.}
Since in the application to cell biology the bulk (cytosolic) diffusion
is much higher than the lateral membrane diffusion a reasonable
reduction of the model can be expected in the limit $D\to\infty$. In the case that the exchange term $q$ satisfies assumption \eqref{eq:ass-q} we deduce by
Proposition \ref{prop:energy-bound} that $\int_0^T\int_B
D|\nabla u|^2$ is bounded uniformly in $D$. The same conclusion holds in the case of any free-energy decreasing evolution by \eqref{EnergyDerivative}. Therefore in the formal
limit $D\to\infty$ we conclude that $u$ is spatially constant and obtain
the (non--local) system of surface PDEs
\begin{alignat}{2}
  \label{eq:nonlocal1}
  \pd_t \varphi &= \SL \mu
  &\qquad\text{on } \Gamma \times (0,T]&,\\
  \label{eq:nonlocal2}
  \mu &= - \eps \SL \varphi + \eps^{-1}W'(\varphi) - \delta^{-1}(2v - 1
- \varphi)
  &\text{on } \Gamma \times (0,T]&,\\
  \label{eq:nonlocal3}
  \pd_t v &= \SL \theta + q = \frac4\delta\SL v - \frac2\delta\SL
  \varphi + q(\varphi,u,v)
  &\text{on } \Gamma \times (0,T].&
\end{alignat}
This system is complemented by initial conditions for $\varphi$ and $v$. The cholesterol concentration $u=u(t)$ is determined by a mass conservation condition
\begin{equation}
  \label{eq:uNonlocal}
  \int_B u(t) + \int_\Gamma v(\cdot,t) = M,
\end{equation}
where $M>0$ is the total mass of cholesterol in the system.

Note that the transformation of the coupled bulk-surface system into a system only defined in the surface has the price of introducing a non-local term by the characterization of $u$ through the mass constraint. The reduction \eqref{eq:nonlocal1}-\eqref{eq:uNonlocal} is similar to the reduction to a shadow system for $2\times 2$ reaction-diffusion systems introduced by Keener \cite{Keen78}, see also the discussion in \cite{Ni11}.

In the qualitative analysis below and the numerical simulations we will
often restrict ourselves to the special choice $q$ of the exchange
function given in \eqref{eq:q-1}. For the reduced model it is then
possible to compute the evolution of the total mass of $u$ and $v$,
which are related by \eqref{eq:uNonlocal}. We deduce, cf.~\eqref{IMB},
\begin{align}
        \label{eq:3.10-a}\frac{d}{dt} \int_B u(t)\,dx \,&=\, -\int_\Gamma q(\varphi(\cdot,t),u(t),v(\cdot,t))\dHa
\,=\, \int_\Gamma -c_1u(t)(1-v)(\cdot,t) + c_2v(\cdot,t) \dHa\\
        &=\, \int_\Gamma (c_1u(t)+c_2)v(\cdot,t)\dHa
-c_1\frac{|\Gamma|}{|B|}\int_B u(t)\,dx \notag\\
        &=\, (c_1u(t)+c_2)\left(M-\int_B u(t)\,dx\right)-c_1\frac{|\Gamma|}{|B|}\int_B
u(t)\,dx \notag\\
        &=\,  -\frac{c_1}{|B|}\left(\int_B u(t)\,dx\right)^2
+\left(c_1\frac{M-|\Gamma|}{|B|}-c_2\right)\int_B u(t)\,dx +c_2M \notag
\end{align}
and therefore we see that $u(t)$ remains nonnegative if it was initially
nonnegative (which is the relevant case) and converges for
$t\to\infty$ to $u_\infty$, which is the positive zero of
\begin{align*}
        p(z) = -c_1z^2 +\big(c_1\frac{M-|\Gamma|}{|B|}-c_2\big)z
+\frac{c_2M}{|B|},
\end{align*}
thus
\begin{align}
        u_\infty \,=\,
\frac{1}{2}\big(\frac{M-|\Gamma|}{|B|}-\frac{c_2}{c_1}\big) +
\sqrt{\frac{1}{4}\big(\frac{M-|\Gamma|}{|B|}-\frac{c_2}{c_1}\big)^2
+\frac{c_2M}{c_1|B|}}. \label{eq:u-infty}
\end{align}
Since $p(0)>0$ and $p(\frac{M}{\left| B \right|})<0$ we also obtain that $\int_\Gamma v(t)$
remains in $[0,M]$ for all times.

We remark that if $q$ is given as in \eqref{eq:q-1}, then assumption \eqref{eq:ass-q} does not hold. Nevertheless we can obtain the reduced model for this particular choice of $q$ if we start with a modified version: Replace first $q$ by \begin{align*}
	\tilde{q}(u,v) \,=\, c_1u -c_1\eta(u)v - c_2v,
\end{align*}
where $\eta:\R\to\R$ is any smooth, monotone increasing and uniformly bounded function with $\eta(r)=r$ for $|r|\leq M|B|^{-1}$. Now the above arguments apply and we obtain the non-local model with exchange term $\tilde q$. By analogous computations as above we then deduce 
\begin{align*}
        \frac{d}{dt} \int_B u(t)\,dx   &=\,  
        -\frac{c_1|\Gamma|}{|B|}\int_B u(t)\,dx + \Big(c_1\eta\big(\frac{1}{|B|}\int_B u(t)\,dx\big) +c_2\Big)\big(M-\int_B u(t)\,dx\big),
\end{align*}
which yields $0\leq \int_B u(t)\,dx\leq M$ for all $t\geq 0$ if this property holds for the initial data. But this implies $\tilde q(u(t),v(\cdot,t))=q(u(t),v(\cdot,t))$ for all $t\geq 0$. This justifies to consider in the following analysis and in the numerical simulations in Section \ref{sec:simulations} the exchange term $q$ from \eqref{eq:q-1} also for the non-local reduction.
\subsection{Stationary points}
\label{sec:statPoints}
We are in particular interested in the long-time behavior of solutions and will therefore next investigate stationary points of our lipid raft system. For the full system \eqref{eq:orig_mu}-\eqref{eq:v} stationary points $(\varphi_\infty,u_\infty,v_\infty)$ are characterized by the equations
\begin{align}
  \notag
  0 &= \SL \mu_\infty\qquad\text{ on }\Gamma,  \\
  \label{eq:3-stat}
  0 &= \SL \theta_\infty + q(\varphi_\infty,u_\infty,v_\infty)\qquad\text{ on }\Gamma,\\
  \notag
  \Delta u_\infty &=\, 0\quad\text{ in }B,\qquad -D\nabla
u_\infty\cdot\nu \,=\,q(\varphi_\infty,u_\infty,v_\infty) \quad\text{ on }\Gamma.
\end{align}
This implies that $\mu_\infty$ is constant and
\begin{align}
        \label{eq:n1-stat}
        -\eps\SL \varphi_\infty + \frac{1}{\eps}W'(\varphi_\infty)
\,&=\, \frac{1}{\delta}(2v_\infty-1-\varphi_\infty) +\mu_\infty\,=\, \frac{\theta_\infty}{2} +\mu_\infty,\\
        q(\varphi_\infty,u_\infty,v_\infty) \,&=\,
-\frac{2}{\delta}\Delta(2v_\infty-1-\varphi_\infty), \label{eq:2a-stat}\\
  \Delta u_\infty &=\, 0\quad\text{ in }B, \qquad -D\nabla
u_\infty\cdot\nu \,=\,q(\varphi_\infty,u_\infty,v_\infty) \quad\text{ on }\Gamma.
\label{eq:2b-stat}
\end{align}
Alternatively, this system can be characterized by
\begin{align}
        -\eps\SL \varphi_\infty +
\frac{1}{\eps}W'(\varphi_\infty) \,=\, \mathcal Q(\varphi_\infty)+\mu_\infty, \label{eq:OKvariant}
\end{align}
where $\mathcal Q(\varphi_\infty)=\frac{1}{2}\theta_\infty$ is a
nonlocal function of $\varphi_\infty$ as $u_\infty,v_\infty$ and hence $\theta_\infty$ are
determined by $\varphi_\infty$ through \eqref{eq:2a-stat},
\eqref{eq:2b-stat}.

\subsection{The case of an energy-decreasing evolution}
\label{subsec:energyDecreasingEvolution}
Let us in the following first consider the case that $q$ is chosen such that the right-hand side in \eqref{EnergyDerivative} is nonpositive, hence the total free energy is decreasing. For any stationary point $u_\infty,v_\infty,\varphi_\infty$ the energy inequality \eqref{EnergyDerivative} yields that $\mu_\infty,\theta_\infty$ and $u_\infty$ are constant, in particular by \eqref{eq:n1-stat}
\begin{align*}
	 -\eps\SL \varphi_\infty + \frac{1}{\eps}W'(\varphi_\infty) \,&=\,\frac{\theta_\infty}{2} +\mu_\infty\in\R,\\
	 q(\varphi_\infty,u_\infty,v_\infty)\,&=\, 0.
\end{align*}
In addition, we fix the total lipid and cholesterol masses as they are for any evolution determined by the initial conditions. We therefore prescribe, for $M,M_1$ given,
\begin{align*}
	M \,&=\, \int_B u_\infty + \int_\Gamma v_\infty  \,=\, |B|u_\infty + \int_\Gamma v_\infty,\\
	M_1 \,&=\, \int_\Gamma \varphi_\infty.
\end{align*}
In particular, the stationary state $\varphi_\infty$ coincides with a critical point of the Cahn--Hilliard energy subject to a volume constraint.
The condition $q=0$ provides an additional relation between $\varphi_\infty$, $u_\infty$ and $v_\infty$. In the case of the exchange law \eqref{TCQ} this determines $\int_\Gamma v_\infty$ and $\theta_\infty$. 

We can elaborate the connection with stationary points of the Cahn--Hilliard equation a bit more if we in addition assume that $(u_\infty,v_\infty,\varphi_\infty)$ is a local minimizer of the energy \eqref{eq:energy}. We represent the latter as
\begin{equation}
  \label{eq:energySplit}
	\F(v, \varphi) = \F_1(\varphi) + \F_2(v, \varphi),
\end{equation}
where
\begin{align*}
  \F_1(\varphi) &\,:=\, \int_\Gamma\frac\eps2|\nabla\varphi|^2 +
  \eps^{-1}W(\varphi),\quad
  \F_2(v, \varphi) \,:=\, \int_\Gamma\frac1{2\delta}(2v - 1 -
  \varphi)^2.
\end{align*}
We also assume that both $\int_\Gamma \varphi_\infty$ and $\int_\Gamma v_\infty$ are fixed by the initial data and the condition $q=0$, which holds in particular in case of the exchange law \eqref{TCQ}.\\
For any $v,\varphi$ with $\int_\Gamma v =\int_\Gamma v_\infty$ and $\int_\Gamma \varphi =\int_\Gamma \varphi_\infty$
we then have
\begin{align*}
  \int_\Gamma (2v - 1- \varphi)^2
  &= \int_{\Gamma} \left(2v - 1- \varphi - \MW(2v - 1- \varphi)\right)^2
  + \int_\Gamma \left(\MW_\Gamma(2v - 1- \varphi)\right)^2\\
  &\geq \int_\Gamma \left(\MW_\Gamma(2v_\infty - 1- \varphi_\infty)\right)^2
\end{align*}
and deduce that under the respective mass constraints the minimizer of $\F_2(v,\varphi)$ are given by those $(v,\varphi)$ for which $2v - 1- \varphi$ is constant, and that the minimum only depends on $\int_\Gamma v$ and $\int_\Gamma \varphi$. Since $\theta_\infty$ and thus $2v_\infty - 1- \varphi_\infty$ are constant, we see that $u_\infty,v_\infty,\varphi_\infty$ minimizes $\F_2$.
\\
Moreover, for any $\varphi$ satisfying the mass constraint the minimum of $\F_2$ is attained by $v=\frac{1}{2}(1+ \varphi + \fint(2v - 1- \varphi))$ and is independent of $\varphi$. Therefore $\varphi_\infty$ as above needs to be a local minimizer of the Cahn--Hilliard energy $\F_1$  subject to a given mass constraint. In the case of the Cahn--Hilliard dynamics in an open convex set in $\R^n$, is is known \cite{StZu98} that stable stationary points converge in the sharp interface limit $\eps\to 0$ to configurations with one connected phase boundary of constant curvature. In analogy, one therefore might expect that for $\eps>0$ sufficiently small the only local minimizer in our lipid raft model for choices $q$ that 
lead to an energy-decreasing evolution are given by configurations with one lipid phase concentrated in a single geodesic ball on $\Gamma$. In particular, such local minimizer do not represent mesoscale pattern like lipid rafts. 
In Section \ref{sec:5.3.4} we present a numerical simulation for energy decreasing dynamics that confirms the expected behavior.
\subsection{The case of the exchange term \eqref{eq:q-1}}
\label{subsec:exchangeTerm}
Let us next discuss the choice of $q$ as given in \eqref{eq:q-1}. The representation \eqref{eq:OKvariant} shows some similarity to the equation for stationary points of the Ohta--Kawasaki model described above. We can make this more transparent in the case of the reduced system \eqref{eq:nonlocal1}-\eqref{eq:uNonlocal}, at least if we presume that the long term behavior of the reduced systems captures the respective behavior of the full system (which means that the order of limits $D\to\infty$ and $t\to\infty$ can be interchanged). Since $u_\infty$ is constant we obtain, writing $\tilde c_1 = c_1u_\infty$ for convenience, that
\begin{align}
        0 \,&=\, \int_\Gamma q(u_\infty,v_\infty) \,=\, |\Gamma|\tilde c_1 - (\tilde
c_1+c_2)\int_\Gamma v_\infty \label{eq:3.19a}
\end{align}
and further
\begin{align*}
        q(u_\infty,v_\infty)&=\, -\frac{\tilde c_1 +c_2}{2}(v_\infty-1-\varphi_\infty) -\frac{\tilde
c_1 +c_2}{2}(1+\varphi_\infty)+\tilde c_1\\
        &=\, -\frac{\delta(\tilde c_1 +c_2)}{4}\theta_\infty - \frac{\tilde c_1
+c_2}{2}(1+\varphi_\infty)+\tilde c_1\\
        &=\, -\frac{\delta(\tilde c_1 +c_2)}{4}\left(\theta_\infty-\fint\theta_\infty\right)
-\frac{\delta(\tilde c_1 +c_2)}{4}\frac{2}{\delta}\fint (v_\infty - 1 -
\varphi_\infty) - \frac{\tilde c_1 +c_2}{2}(1+\varphi_\infty) +\tilde c_1\\
        &=\, -\frac{\delta(\tilde c_1 +c_2)}{4}\left(\theta_\infty-\fint\theta_\infty\right) -
\frac{\tilde c_1 +c_2}{2}\left(\varphi_\infty-\fint\varphi_\infty\right),
\end{align*}
where we have used \eqref{eq:3.19a} in the third line.
This implies by \eqref{eq:3-stat}
\begin{align*}
        \Big(-\SL +\frac{\delta(\tilde c_1
+c_2)}{4}\Big)\left(\theta_\infty-\fint\theta_\infty\right) \,=\, -\frac{\tilde c_1
+c_2}{2}\left(\varphi_\infty-\fint\varphi_\infty\right).
\end{align*}
Since $\mu_\infty$ is constant we deduce from equation \eqref{eq:OKvariant}
\begin{align*}
	- \eps \SL \varphi_\infty + \eps^{-1}W'(\varphi_\infty) \,&=\, \mu_\infty +\frac{1}{2}\fint\theta_\infty +\frac{1}{2}\left(\theta_\infty-\fint\theta_\infty\right)\\
		\,&=\, \mu_\infty +\frac{1}{2}\fint\theta_\infty -\frac{\tilde c_1 +c_2}{4} \Big(-\SL
+\frac{\delta(\tilde c_1 +c_2)}{4}\Big)^{-1}\left(\varphi_\infty-\fint\varphi_\infty\right).
\end{align*}
For $\delta\ll 1$ this can be approximated by
\begin{align}
	- \eps \SL \varphi_\infty + \eps^{-1}W'(\varphi_\infty) +\frac{\tilde c_1 +c_2}{4} (-\SL)^{-1}\left(\varphi_\infty-\fint\varphi_\infty\right)\,&=\, \mu_\infty +\frac{1}{2}\fint\theta_\infty. \label{eq:OK-tilde}
\end{align}
The latter equation corresponds to stationary points of the
Ohta--Kawasaki functional
\begin{align}
  \label{eq:energyOK}
  \F_{OK}(\varphi) = \int_\Gamma \Big(\frac\eps2|\nabla\varphi|^2 +
  \eps^{-1}W(\varphi)\Big) +
\frac{\sigma}{2}\|\varphi-\fint\varphi\|_{H^{-1}}^2
\end{align}
with 
\begin{equation}
  \label{eq:OKParam1}
  \sigma\,=\, \frac{\tilde c_1 +c_2}{4}, 
\end{equation}
where the constant on the right-hand side of \eqref{eq:OK-tilde} should be interpreted as a Lagrange multiplier associated to a mass constraint for $\varphi$.

The total lipid mass $\int\varphi$ is given by the initial data. We can also identify
$\tilde c_1 +c_2$ as a function of the data. First its value is characterized by $u_\infty$ and using
\eqref{eq:u-infty} we deduce
\begin{align}
  \label{eq:OKParam2}
        \tilde c_1 +c_2 \,=\, \frac{1}{2} \Big(c_2 +\frac{M}{|B|}c_1
-\frac{|\Gamma|}{|B|}c_1\Big) + \sqrt{\frac{1}{4}\Big(c_2
+\frac{M}{|B|}c_1 -\frac{|\Gamma|}{|B|}c_1\Big)^2 +
c_1c_2\frac{|\Gamma|}{|B|}}.
\end{align}
In Section \ref{subsubsec:OhtaKawasaki} we will present simulations that confirm that the long-time behavior of the reduced system is for $\delta\ll 1$ in fact very close to that of the Ohta--Kawasaki dynamics. In particular, in contrast to any choice of $q$ that induces an energy-decreasing evolution, in the case of the exchange term \eqref{eq:q-1} we in fact see the occurrence of mesoscale patterns.
\begin{remark}
Let us highlight one key difference in the long-time behavior of our model in the different cases considered above. For a free-energy decreasing evolution stationary states are characterized by the properties that $\theta_\infty$ is constant and $q_\infty=q(\varphi_\infty,u_\infty,v_\infty)$ zero. For the choice \eqref{eq:q-1} of the exchange term on the other hand and the reduced system we have stationary states with non-vanishing $q_\infty$, which correspond to an persisting exchange of cholesterol between bulk and cell membrane. This process eventually allows for the formation of complex pattern on a mesoscopic scale.
\end{remark}
%==========================================
% sharp interface limit
%==========================================
\section{Sharp interface limit $\eps\to 0$ by formally matched asymptotic expansions}\label{sec:sharpinterface}
In this section we formally derive the sharp interface limit of the diffuse interface model \eqref{eq:orig_mu}--\eqref{eq:v} as $\eps \rightarrow 0.$ We assume throughout this section that the tuple $\left(u_\eps,\varphi_\eps,v_\eps,\mu_\eps,\theta_\eps\right)$ solves \eqref{eq:orig_mu}--\eqref{eq:v} and converges formally as $\eps \rightarrow 0$ to a limit $\left(u,\varphi,v,\mu,\theta\right)$. Furthermore, we suppose that the family of the zero level sets of the functions $\varphi_\eps$ converges as $\eps \rightarrow 0$ to a sharp interface. This interface, at times $t\in [0,T]$ is supposed to be given as a smooth curve $\gamma(t)\subset \Gamma$ which separates the regions $\lbrace \varphi(\cdot,t) = 1 \rbrace$ and $\lbrace \varphi(\cdot,t) = -1 \rbrace.$ 
By a formal asymptotic analysis, we conclude that the limit functions $\left(u,\varphi,v,\mu,\theta\right)$ can be characterized as the solutions of a free boundary problem on the surface $\Gamma,$ which is again coupled to a diffusion equation in the bulk $B.$ Other examples for this method and more details on formal asymptotic analysis can be found in \cite{AGG,AHM,CF,Fi88,GS06} which is by far not a comprehensive list of references. 

The obtained limit problem describes a time-dependent partition of the surface $\Gamma$ into different phase regions
\begin{align}
	\Gamma^+(t) \,:=\, \lbrace \varphi(\cdot,t) =  1\rbrace\quad\text{ and }\quad \Gamma^-(t) \,:=\, \lbrace \varphi(\cdot,t) = -1 \rbrace \label{eq:def-phases}
\end{align}
and the dynamic of the interface $\gamma(t)$ between these two regions. We denote the corresponding separation of the surface-time domain $\Gamma\times(0,T]$ by
\begin{align*}
	\GammaT^\pm \,:=\, \{(x,t)\in \Gamma\times (0,T]\,:\, x\in\Gamma^\pm(t)\}.
\end{align*}
The limit problem then takes the following form.

\begin{siproblem}\label{siproblem}
The sharp interface model obtained from the formal asymptotic analysis is given by 
 \begin{alignat}{2}
\varphi &= \pm 1 & \text{on } \GammaT^{\pm} &,\label{eq:sharp_a}\\
\partial_t u &= D \Delta u &\qquad \text{in } B \times (0,T]&,\label{eq:sharp_b}\\
  - D \nabla u \cdot \nu & = q := c_1 u (1 - v) - c_2 v & \text{on } \Gamma\times (0,T] &,\label{eq:sharp_c}\\
  \Delta_\Gamma \mu &= 0 & \text{on } \GammaT^{\pm}&,\label{eq:sharp_d}\\
  \partial_t v &= \Delta_\Gamma \theta + q & \text{on } \GammaT^{\pm}&,\label{eq:sharp_e}\\
  \theta &= \frac{2}{\delta} \left(2v-1 \mp 1\right) & \text{on } \GammaT^{\pm}&,\label{eq:sharp_f}\\
   2 \mu + \theta&= c_0 \kappa_g & \text{on } \gamma &,\label{eq:sharp_g}\\
   \jump{\mu} &= 0 & \text{on } \gamma &,\label{eq:sharp_h}\\
   \jump{\theta} &= 0 & \text{on } \gamma &,\label{eq:sharp_i}\\
   - 2 \mathcal{V} &= \jump{\nabla_\Gamma \mu}\cdot\nu_\gamma & \text{on } \gamma &,\label{eq:sharp_j}\\
   - \mathcal{V} &= \jump{\nabla_\Gamma \theta}\cdot\nu_\gamma & \text{on } \gamma&, \label{eq:sharp_k}
\end{alignat}
where $\jump{\cdot}$ is the jump across the interface $\gamma$ and $\nu_\gamma(x_0,t_0) \in T_{x_0}\Gamma$ denotes the unit normal to $\gamma(t_0)$ in $x_0\in\gamma(t_0)$, pointing inside $\Gamma^+(t_0)$. The geodesic curvature of $\gamma(t)$ in $\Gamma$ is denoted by $\kappa_g(\cdot,t)$ and $\mathcal{V}(x_0,t_0)$ denotes the normal velocity of 
$\gamma(t_0)$ in $x_0 \in \gamma(t_0)$ in direction of $\nu_\gamma(x_0,t_0)$. For its precise definition, let  
$\gamma_t: U \to \gamma(t) \subset \Gamma$, $t\in (t_0-\delta,t_0+\delta)$ be a smoothly evolving family of local parameterizations of the curves $\gamma(t)$ by arc length over an open interval  $U \subset \mathbb{R}$ and let $\gamma_{t_0}(s_0)=x_0$. Then the normal velocity in $(x_0,t_0)$ is given by
\[ 
	\mathcal{V}(x_0,t_0) =  \frac{d}{dt}\Big|_{t_0}\gamma_{t}(s_0) \cdot \nu_\gamma(x_0,t_0), 
\]
see also \cite{DDE}.
\end{siproblem}

We first deduce the existence of transition layers between the phase regions. By assuming that the functions $\left(u_\eps,\varphi_\eps,v_\eps,\mu_\eps,\theta_\eps\right)$ admit suitable expansions with respect to the parameter $\eps$ in the transitions layers and in the regions away from the interface respectively, we can then deduce that the limit functions at least formally need to fulfill \eqref{eq:sharp_a}--\eqref{eq:sharp_k}.   

\subsection{Asymptotic analysis: Existence of transition layers and outer expansion}\label{subsec:outer_exp}
We start our analysis by expanding the solutions to the coupled model in the outer regions, where $\varphi_\eps$ attains values away from zero. We assume that in these regions all functions in \eqref{eq:orig_mu}--\eqref{eq:v} have expansions of the form
\[ f_\eps = \sum_{k=0}^\infty \eps^k f_k \] where $f_\eps = \varphi_\eps, v_\eps,\ldots,$ etc.

Since we postulated the existence of different phase regions characterized by the values of the limit function $\varphi$, we should first address the existence of these phase regions. To this end, we collect all terms of order $\eps^{-1}$ in \eqref{eq:orig_mu} and obtain
\[ W'(\varphi_0) = 0 \]
and as a consequence we obtain $\varphi_0 = \pm 1$ as the only stable solutions. Since $\varphi_0$ is the dominant term in the assumed expansion as $\eps \rightarrow 0$, we deduce $\varphi_\eps \rightarrow \pm 1$ as $\eps \rightarrow 0$ and the existence of the claimed transition layers. This justifies \eqref{eq:def-phases} and \eqref{eq:sharp_a}.

The discussion of equations \eqref{eq:orig_theta} and \eqref{eq:diffU}--\eqref{eq:v} is then straightforward. Comparing the terms of order $\mathcal{O}(1)$ in their corresponding equations allows us to deduce \eqref{eq:sharp_b}--\eqref{eq:sharp_f}.

Due to the (possibly steep) transition between the regions $\Gamma^+$ and $\Gamma^-$ near the interface, the spatial derivatives occurring in the system might contribute terms which are not necessarily of order $\mathcal{O}(1)$ (with respect to $\eps$) in a neighborhood of the interface. This motivates the need for a more detailed analysis of the functions in the neighborhood of the interface which we will address in Section \ref{subsec:inner_exp}.

\subsection{Coordinates for a neighborhood of the interface}
As stated above, we suppose that the zero level sets of $\varphi_\eps$ converge to some (smooth) curve $\gamma(t)$ with inner (wrt.~$\Gamma^+(t)$) unit normal field $\nu_\gamma(\cdot,t)$. 
We then introduce on a small tubular neighborhood $N$ of $\gamma(t)$ a new coordinate system which is more suitable for the analysis in the transition layer. We remark that the construction of these coordinates presented here is more complicated which is due to the fact that $N$ is a neighborhood of $\gamma(t)$ in the manifold $\Gamma.$ For a similar example, we refer the reader to \cite{ES}.

As above, let $\gamma_t: U \to \gamma(t)$ be a local parametrization of the curve $\gamma(t)$ by arc length over an open interval  $U \subset \mathbb{R}.$ It is then possible to extend $\gamma_t$ to a local parametrization $\Psi_t$ of $N$ by means of the exponential map from differential geometry. 

While details on this map can be found in the literature \cite{DC76,DC92}, it is sufficient for our purpose to quickly recall its definition. For a given point $p \in  \Gamma$ and a vector $\vec a \in T_p\Gamma,$ there exists a unique geodesic curve $c_{\vec a}$ such that $c_{\vec a}(0)=p$ and $c_{\vec a}'(0)=\vec a.$ The exponential map $\exp_p$ in $p$ is then defined for all $\vec a \in T_p\Gamma$ for which $c_{\vec a}(1)$ exists and is given by $\exp_p(\vec a)=c_{\vec a}(1)$. Note that for $z \in [0,1]$ one can easily check $\exp_p(z \vec a)=c_{\vec a}(z).$  

The distance between a point $x \in \Gamma$ and the interface $\gamma(t)$ is defined as
\[ d(x,t) := \inf \lbrace l(c) \left| c:I\rightarrow \Gamma, c \text{ connects } x \text{ and } \gamma(t) \right. \rbrace \]
where $l(c)$ denotes the length of the curve $c.$
Setting
\[ \Psi_t(s,r)=\exp_{\gamma_t(s)}(r\nu_\gamma(\gamma_t(s),t)), \]
we obtain that $\Psi_t$ is a parametrization of a neighborhood $N(t)$ of $\gamma(t).$ If we choose $N(t)$ small enough, the properties of the exponential map imply that  $r$ is the signed distance between the point $\Psi_t(s,r)$ and $\gamma(t),$ that is for $(x,t) = \Psi_t(s,r) \in N(t)$ we have
\[ r = \widehat{d}(x,t) :=\begin{cases}
	  d(\Psi_t(s,r),t) \qquad \text{ if } r \in \Gamma^+(t) \\
	  -d(\Psi_t(s,r),t) \qquad \text{ if } r \in \Gamma^-(t).
       \end{cases}
\]
For the asymptotic analysis, it is necessary to adapt the parametrization to the length scale of the transition layers. We therefore use the rescaled parametrization
\begin{equation}\label{eq:DefParam}
 \Lambda(s,z,t) = \Psi_t(s,\eps z)
\end{equation}
 of $N(t),$ where $z = \frac{r}{\eps}.$ In particular, $z$ can be written as a function of $x$ and $t$. 
 
 Let us remark that as $\gamma(t)$ is the zero level set of the signed distance function  $\widehat{d},$ the tangent space of $\gamma(t)$ is the subspace of the tangent space of $\Gamma$ which is orthogonal to the surface gradient $\SG \widehat{d}$ and thus the normal vector $\nu_\gamma$ of $\gamma(t)$ is given by $\frac{\SG \widehat{d}}{\normg{\SG \widehat{d}}}.$ Because of 
 \[ 0 = \frac{d}{dt} \widehat{d}(\gamma_t(s),t) = \SG \widehat{d}(\gamma_t(s),t) \cdot \partial_t \gamma_t(s) + \partial_t \widehat{d}(\gamma_t(s),t) \] we can thus compute \[\partial_t \gamma_t(s)\cdot \nu = -\frac{\partial_t \widehat{d}(\gamma_t(s),t)}{\normg{\SG \widehat{d}(\gamma_t(s),t)}} \] and therefore the time derivative $\partial_t z$ on $\gamma(t)$ fulfills 
 \[ \partial_t z = - \eps^{-1} \mathcal{V}. \]

\begin{remark}[Gradient, Divergence and Laplace Operator in the new coordinates]\label{rem:newCoord}
From the definition of $\Lambda$ in \eqref{eq:DefParam} we see 
\begin{align}
\Lambda(s,0,t)&=\gamma_t(s), \\ 
\partial_s \Lambda(s,0,t) &= \partial_s \gamma_t(s) = \gamma_t'(s) \text{ and }\label{eq:ParamTang} \\ 
\partial_z \Lambda(s,0,t) &= \eps \nu_\gamma(\gamma_t(s))\label{eq:ParamNorm}.
\end{align}
Furthermore, the curve $r \mapsto \Psi_t(s,r)$ is geodesic by definition and hence we have \[ \mathcal{P} \partial_{zz} \Lambda(s,z,t) = 0 \] where $\mathcal{P}$ is the projection on the tangent space $T_{\Lambda(s,z,t)}\Gamma$ on $\Gamma$ in $\Lambda(s,z,t)$.

These observations allow us to calculate
\begin{align*}
\partial_z \left(\partial_s \Lambda(s,z,t) \cdot \partial_z \Lambda(s,z,t)\right) &= \partial_{zs} \Lambda(s,z,t) \cdot \partial_z \Lambda(s,z,t) + \partial_s \Lambda(s,z,t) \cdot \partial_{zz} \Lambda(s,z,t) \\ 
&= \frac{1}{2} \partial_s \left|\partial_z \Lambda(s,z,t)\right|^2 = 0 
\end{align*}
since $\left|\partial_z \Lambda(s,z,t)\right|^2 = \eps^2$ by definition. The equations \eqref{eq:ParamTang} and \eqref{eq:ParamNorm} imply $\partial_s \Lambda(s,0,t) \cdot \partial_z \Lambda(s,0,t) = 0,$ which yields  
\begin{equation}
 \partial_s \Lambda(s,z,t) \cdot \partial_z \Lambda(s,z,t) = 0.
\end{equation}

For simplification of the following calculations, we denote the variables $s$ and $z$ by $s_1$ and $s_2$ respectively. Given the arguments above, the metric tensor with respect to $\Lambda$ is given by 
\begin{align*}
    g_{11}=g_{ss} &= \partial_s \Lambda \cdot \partial_s \Lambda, \\
    g_{12}=g_{21}=g_{sz} &= g_{zs} = \partial_s \Lambda \cdot \partial_z \Lambda = 0 \text{ and }\\
    g_{22}=g_{zz} &= \partial_z \Lambda \cdot \partial_z \Lambda = \eps^2.  
 \end{align*}
 The matrix \[ G:=\begin{pmatrix}
g_{11} & g_{12} \\
g_{21} & g_{22}
\end{pmatrix} \] is thus diagonal and as usual we denote the entries of its inverse $G^{-1}$ by $g^{ij}$. 

For a scalar function $h(x,t)=\widehat{h}(s(t,x),z(t,x),t)$ on $N$, the surface gradient on $N \subset \Gamma$ is hence expressed by
\begin{align}\label{eq:grad_newCoord}
 \SG h = \sum_{i,j=1}^2 g^{ij}\partial_{s_i}\widehat{h}\partial_{s_j}\Lambda = g^{11}\partial_{s_1}\widehat{h}\partial_{s_1}\Lambda+\frac{1}{\eps^2}\partial_{s_2}\widehat{h}\partial_{s_2}\Lambda = \nabla_{\gamma_\eps}\widehat{h}+\frac{1}{\eps}\partial_z\widehat{h}\partial_z\Lambda 
\end{align}
where $\nabla_{\gamma_\eps}$ is the surface gradient on $\gamma_{\eps}=\lbrace \Lambda(s,z,t)|s\in U\rbrace$ for a fixed $z \in [0,1].$ Similarly, 
\begin{align}\label{eq:div_newCoord}
 \SG \cdot a = \nabla_{\gamma_\eps} \cdot \widehat{a} + \frac{1}{\eps}\partial_z \widehat{a} \cdot \partial_z\Lambda
\end{align}
for some vector valued function $a(x,t)=\widehat{a}(s(x,t),z(x,t),t).$ 

In analogy to the appendix in \cite{AGG}, we calculate the Laplace-Beltrami operator $\SL$ in the new coordinates. Due to the properties of the parametrization $\Lambda(s,z,t)$ which already lead to the diagonal structure of the metric tensor above, we find $\nabla_{\gamma_\eps} \widehat{h} \cdot \partial_z \Lambda = 0$ and $\nabla_{\gamma_\eps} \widehat{h} \cdot \partial_{zz} \Lambda = 0.$ Hence
\[ \left( \partial_z \nabla_{\gamma_\eps} \widehat{h} \right) \cdot \partial_z \Lambda = \left( \partial_z \nabla_{\gamma_\eps} \widehat{h} \right) \cdot \partial_z \Lambda + \nabla_{\gamma_\eps} \widehat{h} \cdot \partial_{zz} \Lambda = \partial_z \left( \nabla_{\gamma_\eps} \widehat{h} \cdot \partial_z \Lambda \right) = 0\] and substituting \eqref{eq:grad_newCoord} in \eqref{eq:div_newCoord} we thus compute
\begin{align}\label{eq:laplace_step1}
 \SL h =&  \Delta_{\gamma_\eps} \widehat{h} + \frac{1}{\eps}\left( \nabla_{\gamma_\eps} \partial_z \widehat{h} \right)\cdot \partial_z \Lambda + \frac{1}{\eps} \partial_z \widehat{h} \nabla_{\gamma_\eps} \cdot \partial_z \Lambda \nonumber \\ &+ \frac{1}{\eps}\left( \nabla_{\partial_z \gamma_\eps} \widehat{h} \right)\cdot\partial_z\Lambda  + \frac{1}{\eps^2}\partial_{zz} \widehat{h}\cdot\partial_z \Lambda +\frac{1}{\eps^2} \partial_z \widehat{h} \partial_{zz} \Lambda \cdot \partial_z \Lambda \nonumber \\ 
 =& \Delta_{\gamma_\eps} \widehat{h} + \frac{1}{\eps} \partial_z \widehat{h} \nabla_{\gamma_\eps} \cdot \partial_z \Lambda + \frac{1}{\eps^2}\partial_{zz} \widehat{h}\cdot\partial_z \Lambda
\end{align}
where we have used the identities above.
Since \[ \partial_{ss} \Lambda \cdot \partial_s \Lambda = \frac{1}{2} \partial_s \left(\partial_s \Lambda \cdot \partial_s \Lambda \right) = 0, \]
the curvature vector $\partial_{ss} \Lambda$ of $\gamma(t)$ (seen as a curve in $\mathbb{R}^3$) is an element in $\spa \left(\partial_z \Lambda,\nu_\Gamma\right)$ where $\nu_\Gamma$ denotes the direction normal to the surface $\Gamma.$ The geodesic curvature $\kappa_g$ of $\gamma(t)$ in $\Gamma$ is therefore given by
\[ \kappa_g = \partial_{ss} \Lambda \cdot \partial_z \Lambda = \partial_s \left( \partial_s \Lambda \cdot \partial_z \Lambda\right) - \partial_s \Lambda \cdot \partial_{sz} \Lambda = - \nabla_{\gamma_\eps} \cdot \partial_z \Lambda. \]

As in \cite{AGG}, one can derive
\begin{align*}
 \nabla_{\gamma_\eps} \widehat{h} = \nabla_{\gamma} \widehat{h} + \mathcal{R}_\eps \\
 \nabla_{\gamma_\eps} \cdot \widehat{h} = \nabla_{\gamma} \cdot \widehat{h} + \mathcal{R}_\eps \\
 \Delta_{\gamma_\eps} \widehat{h} = \Delta_{\gamma} \widehat{h} + \mathcal{R}_\eps
\end{align*} where $\mathcal{R}_\eps$ is of higher order in $\eps.$
Thus \eqref{eq:laplace_step1} reads
\begin{align}
 \SL h &= \Delta_{\gamma_\eps} \widehat{h} - \frac{1}{\eps} \kappa_g \partial_z \widehat{h} + \frac{1}{\eps^2}\partial_{zz} \widehat{h}\cdot\partial_z \Lambda \nonumber \\
 &= \Delta_{\gamma} \widehat{h} - \frac{1}{\eps} \kappa_g \partial_z \widehat{h} + \frac{1}{\eps^2}\partial_{zz} \widehat{h}\cdot\partial_z \Lambda + \mathcal{R}_\eps.
\end{align}

\end{remark}

\subsection{Asymptotic analysis: Inner expansion}\label{subsec:inner_exp}

We assume now that the functions in \eqref{eq:orig_mu} - \eqref{eq:v} have inner expansions on $N$ with respect to the new variables of the form
\[ f_\eps(x,t) = F(z,s,t;\eps) = \sum_{k=0}^\infty \eps^k F_k(z,s,t) \]
where again $f_\eps = \varphi_\eps, v_\eps,\ldots$ etc. Accordingly, the inner expansions for $(v,\phi,\mu,\theta)$ will be denoted by $(V,\Phi,M,\Theta)$.

In order for the inner and outer expansions to be consistent which each other, we prescribe the following matching conditions as $z \rightarrow \pm \infty$ (again we use $F$ as an placeholder for $(V,\Phi,M,\Theta)$)
\begin{align}
F_0(t,s,\pm \infty) &\sim f^\pm_0(x,t), \label{cond:match_a} \\
\partial_z F_0(t,s,\pm \infty) &\sim 0, \label{cond:match_b} \\
\partial_z F_1(t,s,\pm \infty) &\sim \nabla_\Gamma f^\pm_0(x,t) \cdot \nu_\gamma \label{cond:match_c}
\end{align}
where $(x,t)= \Lambda(0,s,t)$ and $f^\pm_0(x,t) = \lim_{\delta \rightarrow 0} f_0(\exp_x(\pm \delta \nu_\gamma),t).$ 
We refer the reader to \cite{CF, GS06} and \cite{Fi88} for a derivation of these matching conditions.

We plug these asymptotic expansions in the equations \eqref{eq:orig_mu}-\eqref{eq:v} and use the results from Remark \ref{rem:newCoord} where necessary. Again we collect all terms with the same order in $\eps$. As we assume that the limes $\eps \rightarrow 0$ exists, we require that the terms associated with the leading orders in $\eps$ cancel out. The terms of order $\eps^{-1}$ in equation \eqref{eq:orig_mu} hence yield the differential equation
\[ - \partial_{zz}\Phi_0 + W'(\Phi_0)=0\]
for $\Phi_0.$ Together with the matching conditions and the
condition $\Phi_0(0) = 0$ we deduce 
\begin{equation}\label{eq:Phi_0}
\Phi_0(z) = \tanh(z) \quad \text{ and } \quad \frac{1}{2} | \partial_z \Phi_0 |^2(z) = W(\Phi_0)(z). 
\end{equation}

Looking at the conditions imposed by terms of order $\eps^{-2}$, equation \eqref{eq:CH1} yields $\partial_{zz} M_0 = 0$ and integrating this equation from $-\infty$ to $z$ implies that $M_0$ is independent from $z$ if we take the matching condition \eqref{cond:match_b} into account. Thus
\[ M_0(z=+\infty) = M_0(z=-\infty), \] which in turn gives \eqref{eq:sharp_h}.

In a similar way, we obtain equation \eqref{eq:sharp_i}. The terms of order $\eps^{-2}$ in \eqref{eq:v} indeed imply $\partial_{zz} \Theta_0 = 0$ and integrating in $z$ yields
\begin{equation}\label{eq:der_Theta}
\partial_z \Theta_0 = 0
\end{equation}
by the matching conditions. Again we deduce $\Theta_0(z=+\infty) = \Theta_0(z=-\infty)$ and \eqref{cond:match_a} yields \eqref{eq:sharp_i}.

Let us observe for later use that $\jump{\theta} = 0$ directly implies 
\begin{equation}\label{obs:jump_v}
\jump{v} = 1
\end{equation}
as we already saw that $\jump{\varphi} = 2$.

A second observation is motivated by the study of all terms of order $\mathcal{O}(1)$ in \eqref{eq:orig_theta}. We can deduce \[ \Theta_0 = \frac{2}{\delta}\left(2 V_0 - 1 - \Phi_0\right)  \] and since $\partial_z \Theta_0 = 0$ by \eqref{eq:der_Theta} we therefore obtain
\begin{equation}\label{cond:der_V_0}
2 \partial_z V_0(z,s,t) = \partial_z \Phi_0(z).
\end{equation}

Substituting the inner expansions in \eqref{eq:CH1} also yields terms of order $\eps^{-1}.$ The resulting equation reads \[ -\mathcal{V} \partial_z \Phi_0 = \partial_{zz} M_1. \] Equation \eqref{eq:sharp_j} is then obtained from the matching conditions by integrating in $z$.

Next, we study the terms of order $\mathcal{O}(1)$ in \eqref{eq:orig_mu}. Similar to the studies on the related Cahn-Hilliard equation we obtain 
\begin{align*}
M_0 &= -\partial_{zz} \Phi_1 + \kappa_g \partial_z \Phi_0 + W''(\Phi_0)\Phi_1 - \frac{1}{\delta}\left(2V_0-1-\Phi_0\right)\\
&=-\partial_{zz} \Phi_1 + \kappa_g \partial_z \Phi_0 + W''(\Phi_0)\Phi_1 - \frac{1}{2}\Theta_0
\end{align*}
and apply the following solvability condition for $\Phi_1$ derived in \cite[Lemma 2.2]{AHM}. 

\begin{lemma}[\cite{AHM}]
\label{l:solvCond}
Let $A(z)$ be a bounded function on $-\infty < z < \infty.$ Then the problem
\begin{align*}
  \partial_{zz} \phi + W''(F_0(z))\phi &= A(z) & z \in \mathbb{R}, \\
  \phi(0) &= 0 & \phi \in L^\infty(\mathbb{R}),
\end{align*}
has a solution if and only if
\begin{equation*}
	\int_\mathbb{R} A(z)F'_0(z) \ dz = 0.
\end{equation*}
\end{lemma}
It is indeed easy to see that this condition is necessary if one multiplies the equation by $F_0'$ and uses integration by parts. The assertion that the condition is also sufficient can be derived from the method of variation of constants, details are given in \cite{AHM}.

Lemma \ref{l:solvCond} directly yields
\[
2 M_0 = c_0 \kappa_g - \frac{1}{2}\int_{-\infty}^\infty \Theta_0 \partial_z \Phi_0 \ dz
\]
where $ c_0 $ is given by \[ c_0 := \int_{-\infty}^\infty |\partial_z \Phi_0|^2 \ dz. \]

Given the fact that $\Theta_0$ is independent of $z$ (see \eqref{eq:der_Theta}) we deduce
\[ \int_{-\infty}^\infty \Theta_0 \partial_z \Phi_0 \ dz = \Theta_0 \int_{-\infty}^\infty \partial_z \Phi_0 \ dz\ = 2 \Theta_0. \]

We thus conclude
\[ 2 \mu + \theta = c_0 \kappa_g\]
as we claimed in \eqref{eq:sharp_g}.

In order to conclude our analysis, we collect all terms of order $\eps^{-1}$ from equation \eqref{eq:v}. We thus have
\[ - \mathcal{V} \partial_z V_0 = \partial_{zz} \Theta_1 \]
and another integration with respect to $z$ allows us to deduce
\[-\mathcal{V} \jump{v} = \jump{\nabla_\Gamma \theta} \cdot \nu_\gamma\]
from the matching conditions. Our observation \eqref{obs:jump_v} on $\jump{v}$ above hence implies \eqref{eq:sharp_k}.

%%%%%%%%%%%%%%%%%% Thermodynamics for sharp interface model

\subsection{Free energy inequalities and conservation properties}
The discussion of the thermodynamical background in Section \ref{sec:thermodynamics} can be extended to the sharp interface model \eqref{eq:sharp_a}--\eqref{eq:sharp_k} derived above if one considers the surface free energy
\begin{align}\label{eq:sharp_surface_energy}
\mathcal{F}(\varphi, v, \gamma) &:= c_0 \int_\gamma 1 \ d\mathcal{H}^1 + \frac{1}{2\delta} \int_\Gamma (2v-1-\varphi)^2 \ d\mathcal{H}^2 \\
&= c_0 \int_\gamma 1 \ d\mathcal{H}^1 + \frac{1}{2\delta} \left[ \int_{\Gamma^-} (2v)^2 \ d\mathcal{H}^2 + \int_{\Gamma^+} (2v-2)^2 \ d\mathcal{H}^2 \right] \nonumber
\end{align}
and the bulk free energy
\begin{align}\label{eq:sharp_bulk_energy}
\mathcal{F}_b(u) := \frac{1}{2} \int_B u^2 \ dx.
\end{align}

In particular, choosing these energies the energy inequality derived in Lemma \ref{l:free_energy_ineq} for the diffuse interface model from the second law of thermodynamics has a counterpart in the sharp interface model, i.e.
\[\frac{d}{dt} \left[ \mathcal{F}(\varphi, v, \gamma) + \mathcal{F}_b(u) \right] \leq \int_\Gamma q(\theta-u) \ d\mathcal{H}^2.\]
As the following calculations show, this is mostly due to the relations specified in the equations \eqref{eq:sharp_g}--\eqref{eq:sharp_k} on $\gamma.$ They come into play since
\begin{equation}\label{eq:derivative_of_length}
\frac{d}{dt} \int_\gamma 1 \ d\mathcal{H}^1 = \int_\gamma \kappa_g \mathcal{V} \ d\mathcal{H}^1. 
\end{equation}
We refer the reader to Section 5.4 in \cite{DC76} for a derivation of this identity.

Let us now calculate the time derivative of \eqref{eq:sharp_surface_energy} in order to derive the desired energy inequality. Reynolds' transport theorem and \eqref{eq:derivative_of_length} imply
\begin{align}\label{eq:sharp_energy_ineq_step1}
\frac{d}{dt}\mathcal{F}(\varphi, v, \gamma) &= c_0 \int_\gamma \kappa_g \mathcal{V} \ d\mathcal{H}^1 + \frac{1}{2\delta} \frac{d}{dt}\int_\Gamma \left(\frac{\delta}{2} \theta\right)^2 \ d\mathcal{H}^2 \nonumber \\
 &= c_0 \int_\gamma \kappa_g \mathcal{V} \ d\mathcal{H}^1 + \frac{\delta}{8} \frac{d}{dt} \left(\int_{\Gamma^+} \theta^2 \ d\mathcal{H}^2+ \int_{\Gamma^-} \theta^2 \ d\mathcal{H}^2\right) \\
 &= c_0 \int_\gamma \kappa_g \mathcal{V} \ d\mathcal{H}^1 + \frac{\delta}{4} \int_{\Gamma^+} \theta \partial_t \theta \ d\mathcal{H}^2 + \frac{\delta}{4} \int_{\Gamma^-} \theta \partial_t \theta \ d\mathcal{H}^2. \nonumber
\end{align}
We have used equation \eqref{eq:sharp_i} to derive the last equality. Because of \eqref{eq:sharp_g}, the first integral $c_0 \int_\gamma \kappa_g \mathcal{V} \ d\mathcal{H}^1$ coincides with $\int_\gamma \left( 2 \mu + \theta \right) \mathcal{V} \ d\mathcal{H}^1.$ Using \eqref{eq:sharp_j} and \eqref{eq:sharp_k} yields
\begin{align}\label{eq:int_curvVelo1}
 c_0 \int_\gamma \kappa_g \mathcal{V} \ d\mathcal{H}^1 =& - \int_\gamma \mu \jump{\SG \mu} \cdot \nu_\gamma \ d\mathcal{H}^1 - \int_\gamma \theta \jump{\SG \theta} \cdot \nu_\gamma \ d\mathcal{H}^1 \nonumber \\
 =& - \int_\gamma \mu\SG\mu^+\cdot\nu_\gamma \ d\mathcal{H}^1 + \int_\gamma \mu\SG\mu^-\cdot\nu_\gamma \ d\mathcal{H}^1. \\&- \int_\gamma \theta\SG\theta^+\cdot\nu_\gamma \ d\mathcal{H}^1 + \int_\gamma \theta\SG\theta^-\cdot\nu_\gamma \ d\mathcal{H}^1.\nonumber
\end{align}
By the divergence theorem for tangential vector fields on manifolds, we see that all integrals equate to integrals over $\Gamma^+$ and $\Gamma^-$ respectively. That is, we find 
\begin{align*}
 - \int_\gamma \mu\SG\mu^+\cdot\nu_\gamma \ d\mathcal{H}^1 =& - \int_{\Gamma^+} \div\left(\mu\SG\mu\right) \ d\mathcal{H}^2 \\ 
 =& - \int_{\Gamma^+} |\SG \mu|^2 \ d\mathcal{H}^2 - \int_{\Gamma^+} \mu \SL \mu \ d\mathcal{H}^2 \\
 =& - \int_{\Gamma^+} |\SG \mu|^2 \ d\mathcal{H}^2.
\end{align*}
The last equation holds because of \eqref{eq:sharp_d}. 
Keeping in mind that the orientation of $\gamma$ seen as the boundary of $\Gamma^+$ differs from the orientation of $\gamma$ seen as the boundary of $\Gamma^-$, similar calculations lead to \[ \int_\gamma \mu\SG\mu^-\cdot\nu_\gamma \ d\mathcal{H}^1 = - \int_{\Gamma^-} |\SG \mu|^2 \ d\mathcal{H}^2.\]

For the other integrals in \eqref{eq:int_curvVelo1}, we use \eqref{eq:sharp_e} and \eqref{eq:sharp_f} to calculate
\begin{align*}
  -\int_\gamma \theta\SG\theta^+\cdot\nu_\gamma \ d\mathcal{H}^1 =& - \int_{\Gamma^+} |\SG \theta|^2 \ d\mathcal{H}^2 - \int_{\Gamma^+} \theta \SL \theta \ d\mathcal{H}^2 \\
  =& - \int_{\Gamma^+} |\SG \theta|^2 \ d\mathcal{H}^2 - \int_{\Gamma^+} \theta \left(\partial_t v - q\right) \ d\mathcal{H}^2 \\
  =& - \int_{\Gamma^+} |\SG \theta|^2 \ d\mathcal{H}^2 - \frac{\delta}{4}\int_{\Gamma^+} \theta \partial_t \theta \ d\mathcal{H}^2 + \int_{\Gamma^+} \theta q  \ d\mathcal{H}^2
\end{align*}
and the analogue result \[\int_\gamma \theta\SG\theta^-\cdot\nu_\gamma \ d\mathcal{H}^1 = - \int_{\Gamma^-} |\SG \theta|^2 \ d\mathcal{H}^2 - \frac{\delta}{4}\int_{\Gamma^-} \theta \partial_t \theta \ d\mathcal{H}^2 + \int_{\Gamma^-} \theta q  \ d\mathcal{H}^2.\]

Plugging these findings in \eqref{eq:int_curvVelo1} yields 
\[ c_0 \int_\gamma \kappa_g \mathcal{V} \ d\mathcal{H}^1 = - \int_{\Gamma} |\SG \mu|^2 \ d\mathcal{H}^2 - \int_{\Gamma} |\SG \theta|^2 \ d\mathcal{H}^2 - \frac{\delta}{4}\int_{\Gamma} \theta \partial_t \theta \ d\mathcal{H}^2 + \int_{\Gamma} \theta q  \ d\mathcal{H}^2\]
and by \eqref{eq:sharp_energy_ineq_step1} we deduce
\begin{align*}
 \frac{d}{dt}\mathcal{F}(\varphi, v, \gamma) =& - \int_{\Gamma} |\SG \mu|^2 \ d\mathcal{H}^2 - \int_{\Gamma} |\SG \theta|^2 \ d\mathcal{H}^2 + \int_{\Gamma} \theta q  \ d\mathcal{H}^2 \\ \leq& \int_{\Gamma} \theta q  \ d\mathcal{H}^2.
\end{align*}
Similarly to the derivation of Lemma \ref{l:free_energy_ineq}, we also find 
\[\frac{d}{dt}\mathcal{F}_b(u) \leq \int_\Gamma u q \ d\mathcal{H}^2.\]
This leads to the free energy inequality for the sharp interface model stated in \ref{siproblem},
\begin{align}\label{eq:sharp_energy_ineq}
 \frac{d}{dt} \left[ \mathcal{F}(\varphi, v, \gamma) + \mathcal{F}_b(u) \right] \leq \int_\Gamma q(\theta-u) \ d\mathcal{H}^2.
\end{align}

Furthermore, equation \eqref{eq:sharp_a} and Reynolds' transport theorem allow us to deduce
\begin{align}\label{eq:sharp_lipid_mb}
\frac{d}{dt} \int_\Gamma \varphi \ d\mathcal{H}^2 &= \frac{d}{dt} \int_{\Gamma^+(t)} 1 \ d\mathcal{H}^2 - \frac{d}{dt} \int_{\Gamma^-(t)} 1 \ d\mathcal{H}^2 \\
&= \int_{\gamma(t)} \mathcal{V}\cdot \nu_\gamma \ d\mathcal{H}^2 - \int_{\gamma(t)} \mathcal{V}\cdot \nu_\gamma \ d\mathcal{H}^2 = 0. \nonumber
\end{align}
We remark that the signs in the above equation depend on the values of $\varphi$ in $\Gamma^+(t)$ and $\Gamma^-(t)$ respectively as well as the fact that $\gamma(t)$ is seen as the boundary of $\Gamma^+(t)$ for the first integral and as the boundary of $\Gamma^-(t)$ for the second integral.  

For energy densities chosen according to \eqref{eq:sharp_surface_energy} and \eqref{eq:sharp_bulk_energy}, equations \eqref{eq:sharp_b} and \eqref{eq:sharp_c} correspond to the mass balance equation for cytosolic cholesterol \eqref{IMB}, while the mass balance equation \eqref{MB2} for the membrane bound cholesterol corresponds to \eqref{eq:sharp_e}. We use these equations to deduce
 \begin{align}\label{eq:sharp_cholesterol_mb}
  \frac{d}{dt} \left[ \int_B u \ dx  + \int_\Gamma v \ d\mathcal{H}^2 \right] = 0.
 \end{align}

Equations \eqref{eq:sharp_lipid_mb} and \eqref{eq:sharp_cholesterol_mb} show that the conservation laws derived in Lemma \ref{l:free_energy_ineq} hold for the sharp interface model as well. Equation \eqref{eq:sharp_energy_ineq} is the analogue  to the general energy inequality in Lemma \ref{l:free_energy_ineq} for the sharp interface model, if one chooses the constitutive relations in such a way that the resulting free energy densities lead to \eqref{eq:sharp_surface_energy} and \eqref{eq:sharp_bulk_energy}.

%
%==========================================
% numerical simulations
%==========================================
\section{Numerical simulations}
\label{sec:simulations}

In this section we present numerical results for the reduced non-local
system \eqref{eq:nonlocal1}--\eqref{eq:uNonlocal} and for the fully
coupled bulk--surface model \eqref{eq:orig_mu}--\eqref{eq:q-1}. For the
first, we propose a discretization which is semi-implicit in time and
which is based on surface finite elements in space
\cite{DzEl07,DzEl13}. The details of the discretization are shown in
Section \ref{subsec:discretization}. We validate the numerical
method in Section  \ref{subsec:benchmarks} with two benchmark tests
and subsequently, in Section  \ref{subsec:simulation}, we present
simulation results showing properties of the model. Moreover, in
Section  \ref{subsec:diffuse}, we present a numerical simulation for a diffuse interface approximation of
the fully coupled bulk--surface model
\eqref{eq:orig_mu}--\eqref{eq:q-1}. For a large bulk diffusion
coefficient, the results show a good agreement with those obtained for
the reduced model in Section  \ref{subsec:simulation}. This further
justifies to use numerical simulations for the reduced model in order
to study the properties of the original model.

Cahn--Hilliard systems on stationary surfaces have been numerically
investigated by parametric finite elements \cite{RaVo06}, level set
techniques \cite{GrBeSa05} and with diffuse interface methods
\cite{RaVo06} mostly applied for the simulation of spinodal
decomposition and coarsening scenarios. For corresponding results on
moving surfaces we refer to \cite{ElSt10,MeMaRiHa13} (sharp interface
approach) and \cite{WaDu08} (diffuse interface approach),
respectively. Recently, a Cahn--Hilliard like system has been
numerically studied for the simulation of the influence of membrane
proteins on phase separation and coarsening on cell membranes
\cite{WiBaVo12}.

\subsection{Surface FEM Discretization}
\label{subsec:discretization}

For the discretization of the reduced system
\eqref{eq:nonlocal1}--\eqref{eq:uNonlocal} with $q$ given by
\eqref{eq:q-1} we develop a scheme similar to the one in \cite{RR},
where a related non-local reaction--diffusion system has been
numerically treated. To discretize in time, we introduce steps
$\tau_m$, $m=1,\ldots,M_T$. Denoting by $\varphi^{(m)}$, $\mu^{(m)}$,
$v^{(m)}$, $u^{(m)}$ the time discrete solution at time $t_m =
\sum\limits_{l=1}^m\tau_l$, the time discrete weak formulation then
reads
\begin{align}
  \label{eq:Weak1}
  & \int_\Gamma\frac{\varphi^{(m+1)}}{\tau_{m+1}}\psi
  + \int_\Gamma \SG \mu^{(m+1)}\cdot \SG
  \psi 
  = \int_\Gamma\frac{\varphi^{(m)}}{\tau_{m+1}}\psi,\\ 
  \label{eq:Weak2}
  &-\int_\Gamma\mu^{(m+1)} \eta 
  + \eps \int_\Gamma\SG \varphi^{(m+1)} \cdot \SG \eta 
  + \frac1\eps  \int_{\Gamma} W''(\varphi^{(m)}) \varphi^{(m +
    1)} \eta
  - \frac2\delta\int_\Gamma v^{(m+1)} \eta \\
  &\quad + \frac1\delta\int_\Gamma \varphi^{(m+1)} \eta
  =   \frac1\eps \int_{\Gamma} ( W''(\varphi^{(m)}) \varphi^{(m)}
  - W'(\varphi^{(m)}) )\eta 
  - \frac1\delta\int_\Gamma  \eta,\nonumber\\
  \label{eq:Weak3}
  & \int_\Gamma\frac{v^{(m+1)}}{\tau_{m+1}}\zeta
  + \frac4\delta\int_\Gamma \SG v^{(m+1)}\cdot \SG
  \zeta 
  - \frac2\delta\int_\Gamma \SG \varphi^{(m+1)}\cdot \SG
  \zeta \\
  \quad& + \int_\Gamma (c_1 u^{(m)} + c_2)v^{(m+1)}\zeta 
  = \int_\Gamma\frac{v^{(m)}}{\tau_{m+1}}\zeta 
  + \int_\Gamma c_1 u^{(m)}\zeta \nonumber
\end{align}
for all $\psi, \eta, \zeta \in H^1(\Gamma)$, where the non-local term
\begin{equation}
  \label{eq:nonLocalDiscrete}
  u^{(m)} = \frac1{|B|}\left(M - \int_\Gamma v^{(m)}\right)
\end{equation}
is treated fully explicitly. 

For the spatial discretization, we introduce a triangulation
$\Gamma_h$ of the surface $\Gamma$, and we use linear surface finite
elements to obtain from 
\eqref{eq:Weak1}--\eqref{eq:nonLocalDiscrete} a linear system of equations for
the coefficients of the discrete solutions $\mu_h^{(m+1)}$,
$\varphi_h^{(m+1)}$, $v_h^{(m+1)}$ with respect to the standard Lagrange
basis. The resulting linear system is solved by a direct solver or a
stabilized bi-conjugate gradient (BiCGStab) method depending on the
number of unknowns. Furthermore, we apply a simple time
adaptation strategy, where time steps $\tau_{m+1}$ are chosen (within
bounds) inversely proportional to 
\begin{equation*}
  \max_{x \in \Gamma_h}
  \frac{|\varphi^{(m)}_h(x) - \varphi^{(m-1)}_h(x)|}{\tau_m},
\end{equation*}
see e.g. \cite{RaVo06}. The above numerical scheme is implemented in the
adaptive FEM toolbox AMDiS \cite{VeVo07}.

\subsection{Benchmark tests}
\label{subsec:benchmarks}

In this section we consider a triangulation $\Gamma_h = S^2_h$ of the
unit sphere $S^2$, and we use the following set of parameters:
\begin{equation}
  \label{eq:parametersBasic}
  c_1 = c_2 = 500,
  \quad \varepsilon = 0.02,
  \quad M = 5 |B| = \frac{20\pi}{3},
  \quad \delta = 0.02.
\end{equation}

\subsubsection{Analytic solution for a prescribed right hand
  side}

In order to verify the validity of the above numerical scheme, we
define right hand sides for \eqref{eq:nonlocal1} and
\eqref{eq:nonlocal3} such that an analytic solution for the resulting
system can be given. To be more precise, we let 
\begin{equation}
  \label{eq:phiAnalytic}
  \varphi(x,t) := \tanh\left(\frac{\vartheta(x) + \beta
      -t}{\sqrt{2}\eps}\right) 
\end{equation}
for $x \in S^2$ and $t \in [0,T)$. Moreover, $\beta$ is an angle
related to the initial value $\varphi(x,0)$, and by $\vartheta =
\vartheta(x)$ we denote the angle given by
\begin{equation*}
  \vartheta(x) := \arccos(x_3) \in [0,\pi]
\end{equation*}
for $x = (x_1, x_2, x_3) \in S^2$. Furthermore, we let $v(x,t) :=
\frac12(1 + \varphi(x,t))$ such that $\theta = 0$ holds. Then we
define $F_1$, $F_2 : S^2 \times [0,T] \to \R$ such that
\begin{alignat}{2}
  \label{eq:nonlocal1F}
  \pd_t \varphi - \SL \mu  &= F_1
  &\qquad\text{on } S^2 \times (0,T]&,\\
  \label{eq:nonlocal2F}
  \mu &= - \eps \SL \varphi + \eps^{-1}W'(\varphi)
  &\text{on } S^2 \times (0,T]&,\\
  \label{eq:nonlocal3F}
  \pd_t v - q(u,v) &= F_2
  &\text{on } S^2 \times (0,T]&
\end{alignat}
hold with $u = u(t)$ given by
\begin{equation}\label{eq:5.9a}
  u = \frac{3}{4\pi}\left(M - \int_{S^2} v\right).
\end{equation}
In this way we obtain expressions for $F_1$ and $F_2$, which are
incorporated in the numerical scheme described in Section
\ref{subsec:discretization}. Finally,  we use the adopted scheme to
find approximate numerical solutions $\varphi_h$, $\mu_h$, $v_h$ to
\eqref{eq:nonlocal1F}--\eqref{eq:5.9a} for initial conditions
for $\varphi_h$, $v_h$ given by corresponding values determined by
\eqref{eq:phiAnalytic} and $v = \frac12(1 + \varphi)$. Note that the expression
for $F_2$ involves an integral of $\varphi$ which is numerically
approximated in polar coordinates in every time step. 

With the usual Lagrange interpolation operator $I_h$, we define the
relative errors 
\begin{align*}
  e_\infty(t) &:= \frac{\|I_h\varphi(\cdot,t) - \varphi_h(\cdot,t)\|_{L^\infty(S_h^2)}}
  {\|\varphi(\cdot,t)\|_{L^\infty(S^2)}},\\
  e_1(t) &:= \frac{\|I_h\varphi(\cdot,t) - \varphi_h(\cdot,t)\|_{H^1(S_h^2)}}
  {\|I_h\varphi(\cdot,t)\|_{H^1(S_h^2)}}\\
\end{align*}
of $\varphi_h(\cdot,t)$ in the $L^\infty(S_h^2)$-, 
$H^1(S_h^2)$-norms, respectively, and we consider their values
for simulations with $T = \frac{\pi}{4}$, $\beta = -\frac{\pi}{4}$ and
different spatial resolutions in Table \ref{tab:errorsRHS}. The results
indicate a linear dependence of the $L^\infty(S_h^2)$- and
$H^1(S_h^2)$-errors, respectively, on the grid-size, as expected for
elliptic problems. As a further illustration of this benchmark, in
Figure \ref{fig:phiRHS}, we see contour plots of $\varphi(\cdot,0)$ and
$\varphi(\cdot,T)$, respectively. We remark that in the following we
use a grid with 98306 vertices for all simulations with spherical
geometry. 

\begin{table}[h]
  \begin{center}
    \begin{tabular}[h]{| c | c | c | c | c | c | c |
        c | c | c | c | c | c | c | c | c | c | c |}
      \hline
      number of vertices & $6146$ & $24578$ & $98306$ &
      $393218$ \\
      \hline
       $e_\infty(T)$ & $0.245005$ & $0.0284765$ & $0.0107314$ & $0.00568034$  \\
      \hline
       $e_1(T)$ & $0.2284131488$ & $0.0747088125$ & $0.0407739003$ & $0.0211030453$ \\
      \hline
    \end{tabular}
    \vskip5pt \caption{\footnotesize Relative errors for simulations
      of \eqref{eq:nonlocal1F}--\eqref{eq:5.9a} with different
      resolutions.}\label{tab:errorsRHS}
  \end{center}
\end{table}

\begin{figure}[here]
  \centerline{
    \hfill
    \subfigure[$\varphi(t = 0)$]{
      \includegraphics*[width=0.2\textwidth]{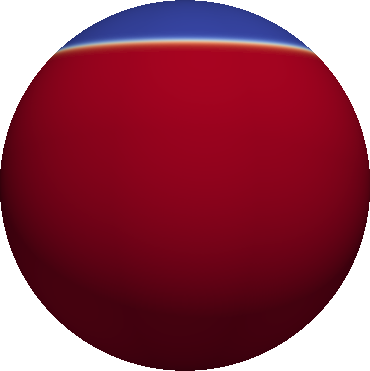}
    }
    \hfill
    \subfigure[$\varphi(t \approx \frac\pi4)$]{
      \includegraphics*[width=0.2\textwidth]{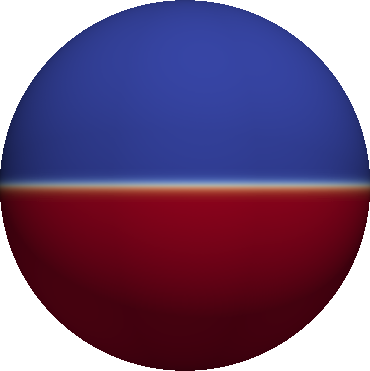}
    }
    \hfill
    \includegraphics*[width=0.08\textwidth]{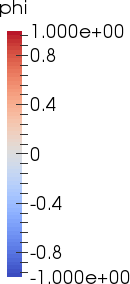}
    \hfill
  }
  \caption{\label{fig:phiRHS}\footnotesize Benchmark simulation with
    prescribed right hand side in
    \eqref{eq:nonlocal1F}--\eqref{eq:nonlocal3F}. Contour plots of
    initial values of analytic solution \eqref{eq:phiAnalytic} and of
    values at time $T$.} 
\end{figure}

\subsubsection{Validation of $\int v$}

We return to \eqref{eq:nonlocal1}--\eqref{eq:uNonlocal} and remark
that one easily verifies from \eqref{eq:3.10-a} that
\begin{equation*}
  z := \int_\Gamma v
\end{equation*}
fulfills the ODE
\begin{equation}
  \label{eq:odeIntV}
  \dot z = c_1 \frac{M |\Gamma|}{|B|} - \left(
    c_1 \frac{|\Gamma|}{|B|} + c_2 + c_1 \frac{M}{|B|}
  \right)z
  + \frac{c_1}{|B|}z^2.
\end{equation}
We compare solutions of the above ODE with
\begin{equation*}
  \int_{S^2_h}v_h
\end{equation*}
obtained by numerical solution of
\eqref{eq:Weak1}--\eqref{eq:nonLocalDiscrete}, where we have chosen
the initial conditions
\begin{equation*}
  \varphi(\cdot,0) = -0.25 +
  \mathcal R, \quad v(\cdot, 0) = 0.25.
\end{equation*}
Thereby, $\mathcal R : \Gamma_h \to [-0.001,0.001]$ provides a
perturbation given by an (irregular and nonperiodic) oscillation around zero.
In Figure \ref{fig:odeIntV}, $z_h =
\int_{S^2_h}v_h$ obtained by numerical simulation of
\eqref{eq:Weak1}--\eqref{eq:nonLocalDiscrete} is compared with a
numerical solution of the ODE \eqref{eq:odeIntV} with initial
condition
\begin{equation*}
  z(0) = \int_{S^2}v(\cdot,0)
\end{equation*}
obtained with Maple\texttrademark \cite{Maple}. The excellent
agreement further justifies the chosen numerical scheme.

\begin{figure}[here]
  \centerline{
    \hfill
    \includegraphics*[width=0.38\textwidth]{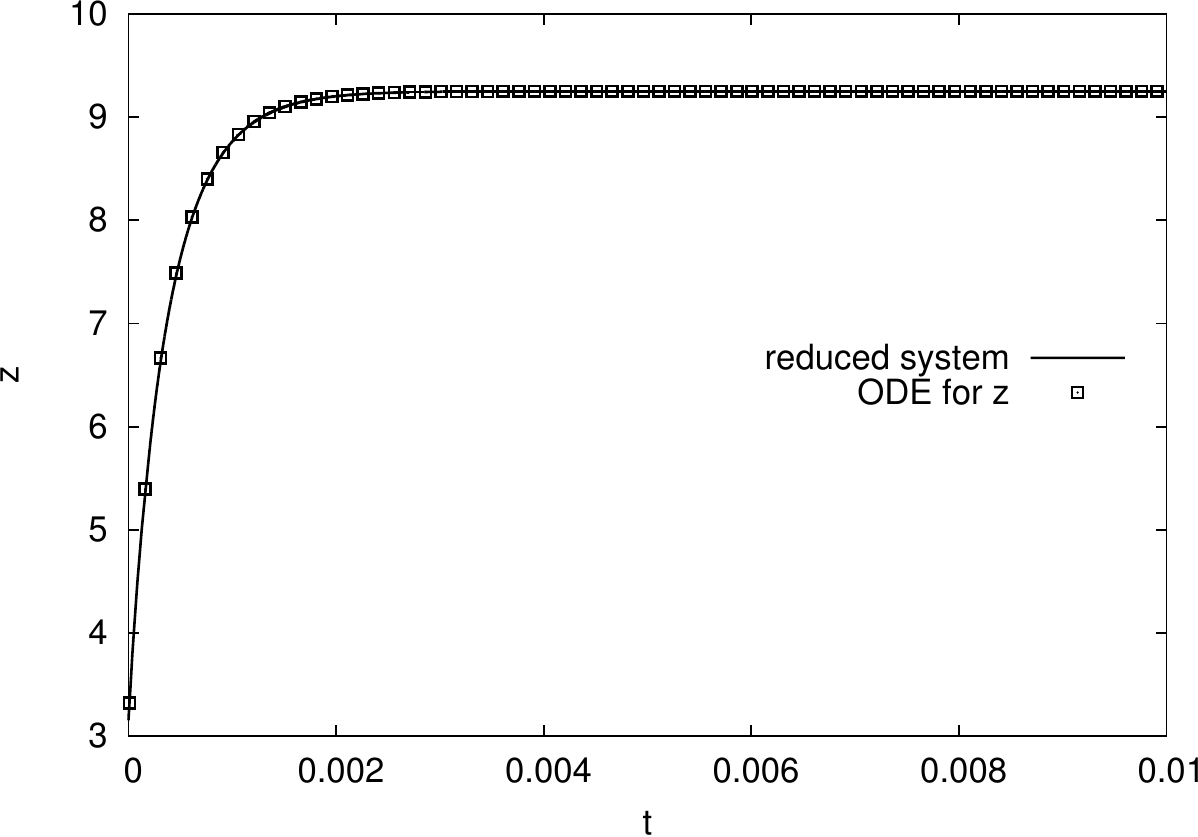}
    \hfill
  }
  \caption{\label{fig:odeIntV}\footnotesize Comparison of a numerical
    solution of the ODE \eqref{eq:odeIntV} with $\int_{S^2_h}v_h$
    obtained by simulation of
    \eqref{eq:nonlocal1}--\eqref{eq:uNonlocal}.}  
\end{figure}

\subsection{Simulation results}
\label{subsec:simulation}

\subsubsection{Variation of parameters}

We use the basic set of parameters given in \eqref{eq:parametersBasic}.
Furthermore, we use the initial conditions $\varphi(\cdot,0) = \hat \varphi +
\mathcal R$, $v(\cdot, 0) = \frac14$, where $\hat \varphi = -0.5$ and $\mathcal R :
\Gamma_h \to [-0.001,0.001]$ again denotes a perturbation by an (irregular and nonperiodic) oscillation around zero. In Figure \ref{fig:basic}, we
see corresponding results, where contour plots of the discrete solutions
$\varphi_h(\cdot, t)$ (first row) and $v_h(\cdot, t)$ (second row) are
displayed for various times $t$. The evolution shows a spinodal
decomposition with subsequent interrupted coarsening scenario.
\begin{figure}[here]
  \centerline{
    \hfill
    \subfigure[$\varphi_h(\cdot,t = 0)$]{
      \includegraphics*[width=0.16\textwidth]{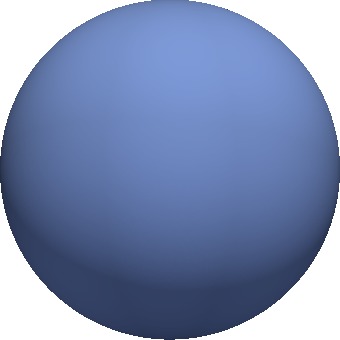}
    }
    \hfill
    \subfigure[$\varphi_h(\cdot,t \approx 0.0036)$]{
      \includegraphics*[width=0.16\textwidth]{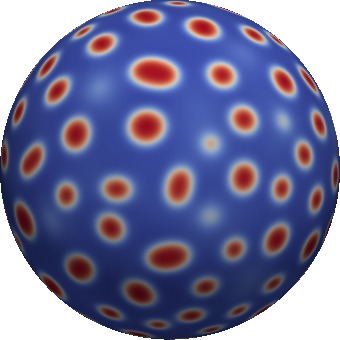}
    }
    \hfill
    \subfigure[$\varphi_h(\cdot,t \approx 0.0107)$]{
      \includegraphics*[width=0.16\textwidth]{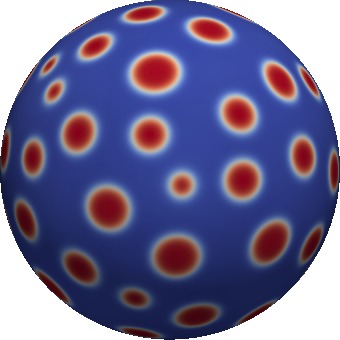}
    }
    \hfill
    \subfigure[$\varphi_h(\cdot,t \approx 0.1177)$]{
      \includegraphics*[width=0.16\textwidth]{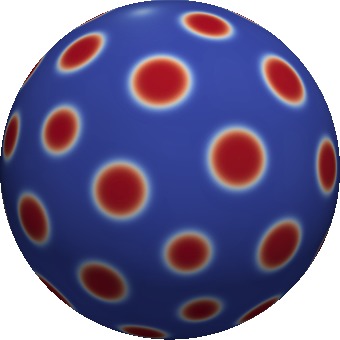}
    }
    \hfill
    \subfigure[$\varphi_h(\cdot,t \approx 402.4228)$]{
      \includegraphics*[width=0.16\textwidth]{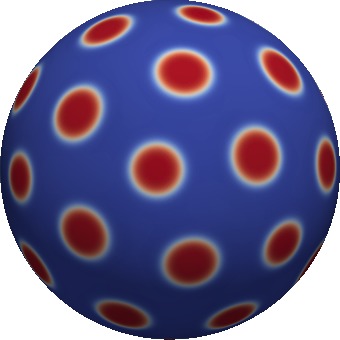}
    }
    \hfill
    \subfigure[]{
      \includegraphics*[height=0.09\textheight]{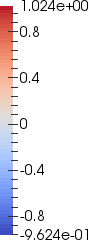}
    }
    \hfill
  }
  \vspace{3ex}
  \centerline{
    \hfill
    \subfigure[$v_h(\cdot,t = 0)$]{
      \includegraphics*[width=0.16\textwidth]{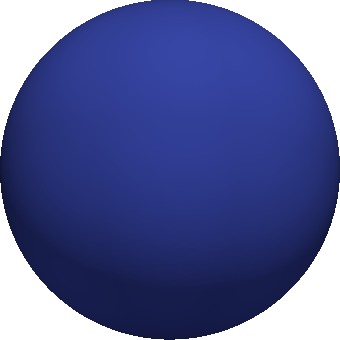}
    }
    \hfill
    \subfigure[$v_h(\cdot,t \approx 0.0036)$]{
      \includegraphics*[width=0.16\textwidth]{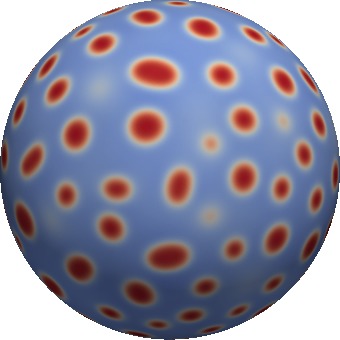}
    }
    \hfill
    \subfigure[$v_h(\cdot,t \approx 0.0107)$]{
      \includegraphics*[width=0.16\textwidth]{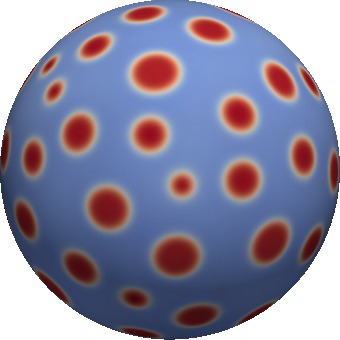}
    }
    \hfill
    \subfigure[$v_h(\cdot,t \approx 0.1177)$]{
      \includegraphics*[width=0.16\textwidth]{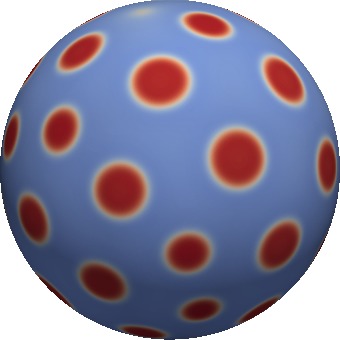}
    }
    \hfill
    \subfigure[$v_h(\cdot,t \approx 402.4228)$]{
      \includegraphics*[width=0.16\textwidth]{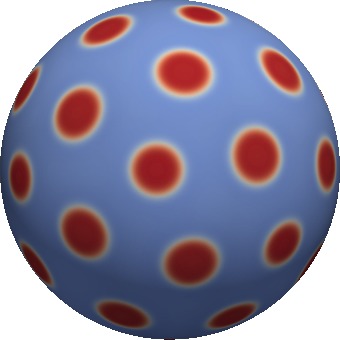}
    }
    \hfill
    \subfigure[]{
      \includegraphics*[height=0.09\textheight]{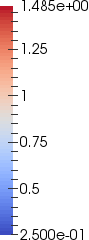}
    }
    \hfill
  }
  \caption{\label{fig:basic}\footnotesize Numerical results for
    $\varphi(\cdot,0) = -0.5 + \mathcal R$. First row:
    contour plots of $\varphi(\cdot,t)$ for several choices of times
    $t$; second row: corresponding contour plots of $v(\cdot, t)$.}
\end{figure}

The geometric shape of the particles can drastically change, if one
changes the values of $\varphi(\cdot,0)$. In a second example, we have
used $\varphi(\cdot,0) = 0 + \mathcal R$, while all
remaining parameters have not been changed. The corresponding
numerical results can be seen in Figure \ref{fig:basic2}.

\begin{figure}[here]
  \centerline{
    \hfill
    \subfigure[$\varphi_h(\cdot,t = 0)$]{
      \includegraphics*[width=0.16\textwidth]{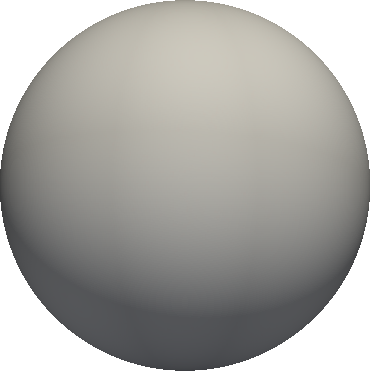}
    }
    \hfill
    \subfigure[$\varphi_h(\cdot,t \approx 0.0038)$]{
      \includegraphics*[width=0.16\textwidth]{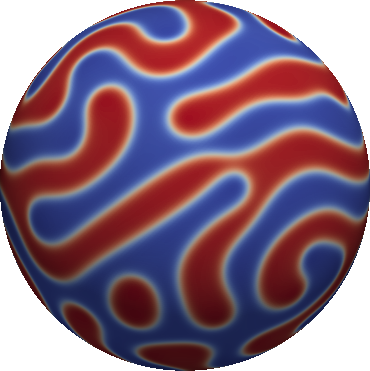}
    }
    \hfill
    \subfigure[$\varphi_h(\cdot,t \approx 0.0125)$]{
      \includegraphics*[width=0.16\textwidth]{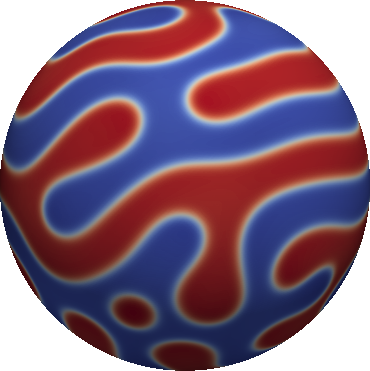}
    }
    \hfill
    \subfigure[$\varphi_h(\cdot,t \approx 0.1035)$]{
      \includegraphics*[width=0.16\textwidth]{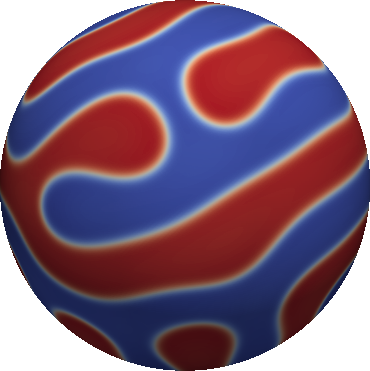}
    }
    \hfill
    \subfigure[$\varphi_h(\cdot,t \approx 379.5465)$]{
      \includegraphics*[width=0.16\textwidth]{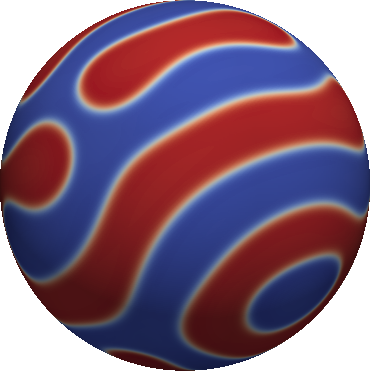}
    }
    \hfill
    \subfigure[]{
      \includegraphics*[height=0.09\textheight]{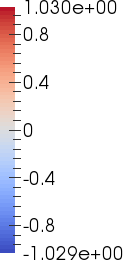}
    }
    \hfill
  }
  \vspace{3ex}
  \centerline{
    \hfill
    \subfigure[$v_h(\cdot,t = 0)$]{
      \includegraphics*[width=0.16\textwidth]{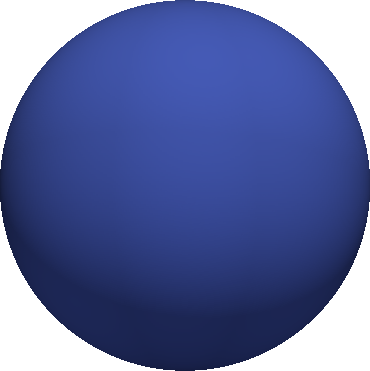}
    }
    \hfill
    \subfigure[$v_h(\cdot,t \approx 0.0038)$]{
      \includegraphics*[width=0.16\textwidth]{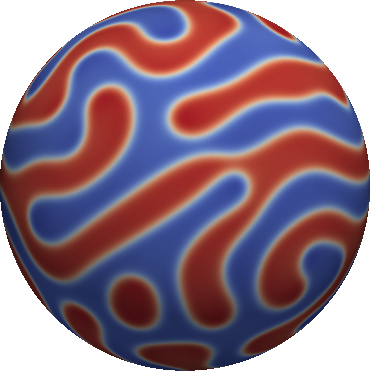}
    }
    \hfill
    \subfigure[$v_h(\cdot,t \approx 0.0125)$]{
      \includegraphics*[width=0.16\textwidth]{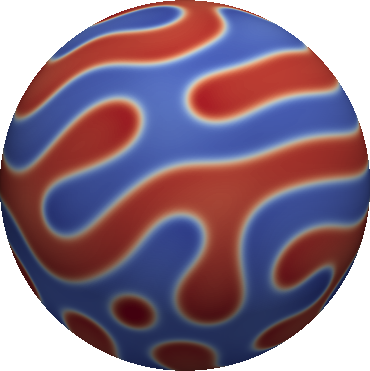}
    }
    \hfill
    \subfigure[$v_h(\cdot,t \approx 0.1035)$]{
      \includegraphics*[width=0.16\textwidth]{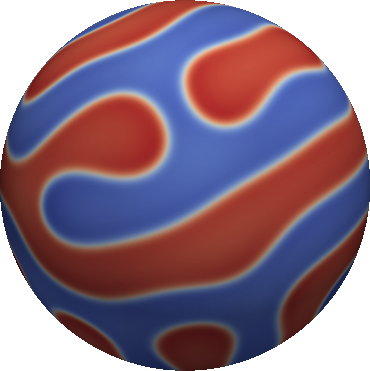}
    }
    \hfill
    \subfigure[$v_h(\cdot,t \approx 379.5465)$]{
      \includegraphics*[width=0.16\textwidth]{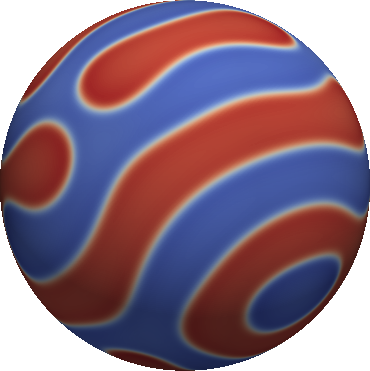}
    }
    \hfill
    \subfigure[]{
      \includegraphics*[height=0.09\textheight]{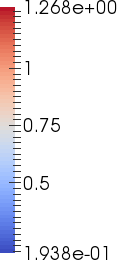}
    }
    \hfill
  }
  \caption{\label{fig:basic2}\footnotesize Numerical results for
    $\varphi(\cdot,0) = 0 + \mathcal R$. First row:
    contour plots of $\varphi(\cdot,t)$ for several choices of times
    $t$; second row: corresponding contour plots of $v(\cdot, t)$.}
\end{figure}

In the following, we compare almost stationary states for simulations
with varying parameters, in order to illustrate the properties
of the model and its solutions. We use the basic set of parameters for
the results shown in Figure \ref{fig:basic} and modify one
parameter. Let
\begin{equation*}
  \bar\varphi := \int_\Gamma \varphi(\cdot,0) \,=\, \int_\Gamma \varphi .
\end{equation*}

\paragraph{\bf Varying $\bar\varphi$} We change $\bar\varphi$ by changing
the initial condition for $\varphi$, or rather $\hat \varphi$. In Figure
\ref{fig:varyPhi} one observes the almost stationary states from the
previous two examples with $\hat\varphi = 0$ and $\hat\varphi = -0.5$
and examples with intermediate values, where one can see a crossover
from circular lipid rafts to stripe like patterns. For $\hat\varphi =
-0.75$ the system allows for a constant stationary state with order
parameter away from the classical phases $\varphi \in \{-1,1\}$.

\begin{figure}[here]
  \centerline{
    \hfill
    \subfigure[$\hat\varphi = 0$]{
      \includegraphics*[width=0.16\textwidth]{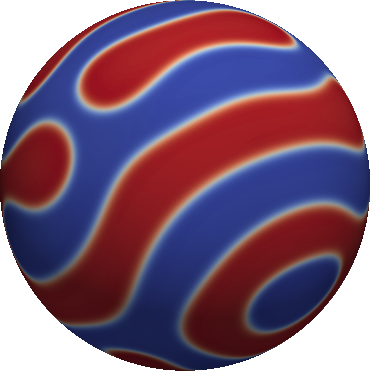}
    }
    \hfill
    \subfigure[$\hat\varphi = -0.1$]{
      \includegraphics*[width=0.16\textwidth]{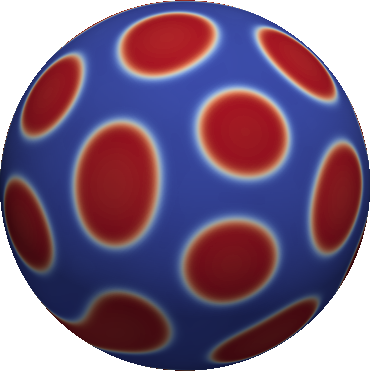}
    }
    \hfill
    \subfigure[$\hat\varphi = -0.25$]{
      \includegraphics*[width=0.16\textwidth]{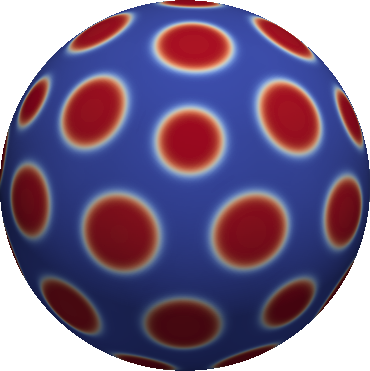}
    }
    \hfill
    \subfigure[$\hat\varphi = -0.5$]{
      \includegraphics*[width=0.16\textwidth]{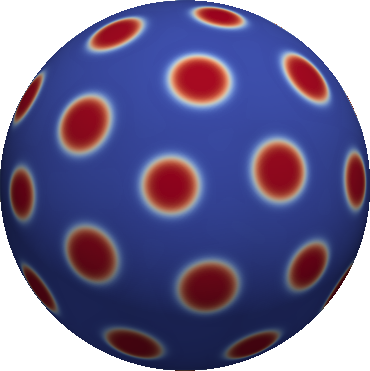}
    }
    \hfill
    \subfigure[$\hat\varphi = -0.75$]{
      \includegraphics*[width=0.16\textwidth]{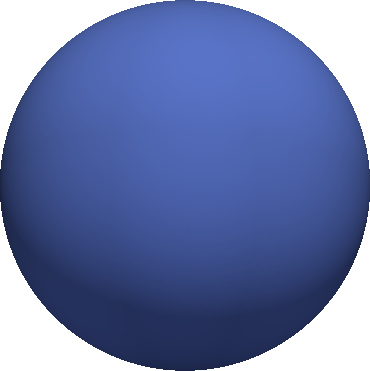}
    }
    \hfill
    \subfigure[]{
      \includegraphics*[height=0.09\textheight]{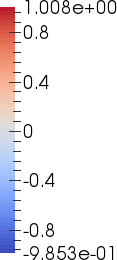}
    }
    \hfill
  }
  \caption{\label{fig:varyPhi}\footnotesize Almost stationary discrete
    solutions $\varphi_h$ for
    different values of $\hat\varphi$.} 
\end{figure}

\paragraph{\bf Varying $\delta$} For $\hat\varphi = -0.5$ we consider
almost stationary discrete order parameter $\varphi_h$ for different
values for the parameter $\delta$. In Figure \ref{fig:varyDelta}, for
large $\delta$ the almost stationary $\varphi_h$ has one lipid raft,
as one would expect for classical Cahn--Hilliard system on the
sphere. For decreasing but still positive values of $\delta$ we expect
to approach stationary solution of Ohta--Kawasaki based dynamics, as
shown in Section \ref{subsec:exchangeTerm}. For a more detailed
comparison with almost stationary states of Ohta--Kawasaki based
dynamics we refer to the subsequent Section \ref{subsubsec:OhtaKawasaki}. 

\begin{figure}[here]
  \centerline{
    \hfill
    \subfigure[$\delta = 0.3$]{
      \includegraphics*[width=0.16\textwidth]{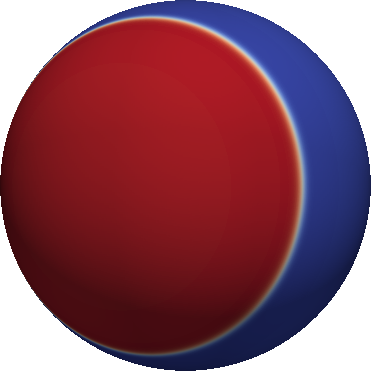}
    }
    \hfill
    \subfigure[$\delta = 0.1$]{
      \includegraphics*[width=0.16\textwidth]{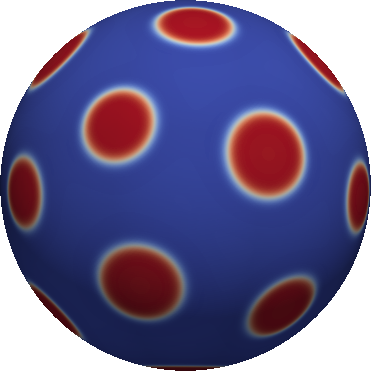}
    }
    \hfill
    \subfigure[$\delta = 0.02$]{
      \includegraphics*[width=0.16\textwidth]{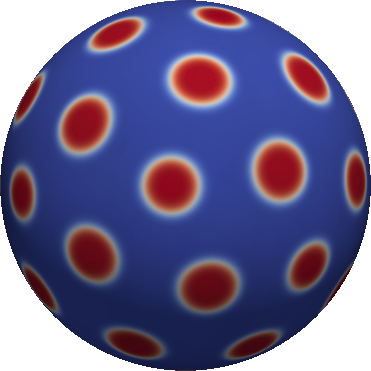}
    }
    \hfill
    \subfigure[$\delta = 0.002$]{
      \includegraphics*[width=0.16\textwidth]{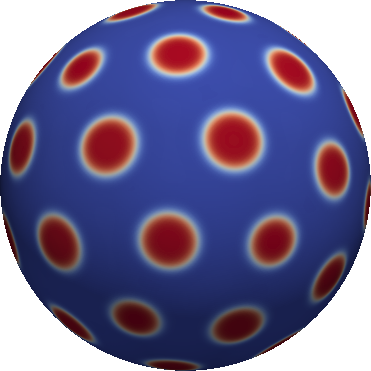}
    }
    \hfill
    \subfigure[]{
      \includegraphics*[height=0.09\textheight]{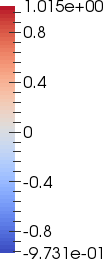}
    }
    \hfill
  }
  \caption{\label{fig:varyDelta}\footnotesize Almost stationary discrete
    solutions $\varphi_h$ for different values of $\delta$.} 
\end{figure}

\paragraph{\bf Varying $c_1$} Returning to the previous standard
parameters $\hat\varphi = -0.5$ and $\delta = 0.02$ we investigate the
influence of different values of $c_1$ on the almost stationary
discrete states. Increasing $c_1$ corresponds to increasing the number of
lipid rafts as shown in Figure \ref{fig:varyC1}. 

\begin{figure}[here]
  \centerline{
    \hfill
    \subfigure[$c_1 = 5$]{
      \includegraphics*[width=0.16\textwidth]{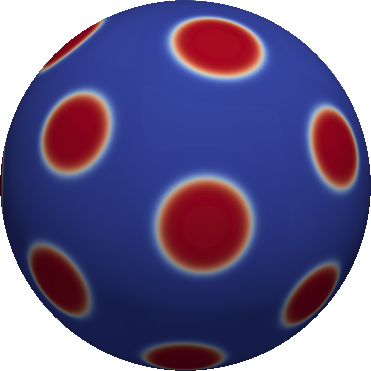}
    }
    \hfill
    \subfigure[$c_1 = 100$]{
      \includegraphics*[width=0.16\textwidth]{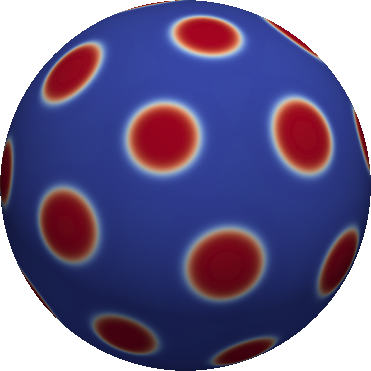}
    }
    \hfill
    \subfigure[$c_1 = 500$]{
      \includegraphics*[width=0.16\textwidth]{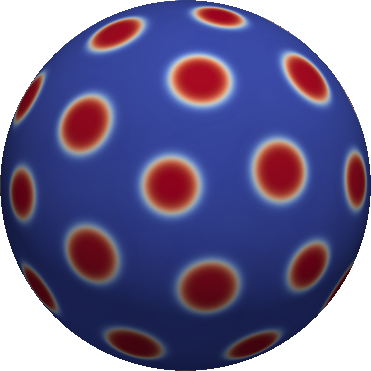}
    }
    \hfill
    \subfigure[$c_1 = 2000$]{
      \includegraphics*[width=0.16\textwidth]{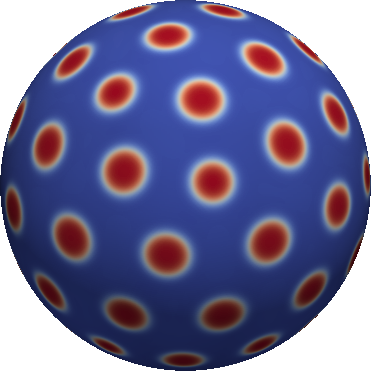}
    }
    \hfill
    \subfigure[]{
      \includegraphics*[height=0.09\textheight]{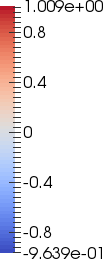}
    }
    \hfill
  }
  \caption{\label{fig:varyC1}\footnotesize Almost stationary discrete
    solutions $\varphi_h$ for different values of $c_1$.} 
\end{figure}

\paragraph{\bf Varying $c_2$} With our standard choice $c_1 = 500$ we
obtain for variation of values for $c_2$ similar results as for the
previous examples (varying $c_1$). In Figure \ref{fig:varyC2}, one
observes increasing number of rafts of decreasing sizes for increasing
$c_2$. 

From the analysis in Section  \ref{subsec:exchangeTerm} we
conclude that increasing $c_1$ or $c_2$ corresponds to increasing the
non-local energy contribution in the Ohta--Kawasaki functional
\eqref{eq:energyOK}, see also \eqref{eq:OKParam1}, which is expected
to describe the dynamics in the limit $\delta \to 0$. This fact
provides an explanation of the increasing number of particles for
increasing $c_1,c_2$.

\begin{figure}[here]
  \centerline{
    \hfill
    \subfigure[$c_2 = 5$]{
      \includegraphics*[width=0.16\textwidth]{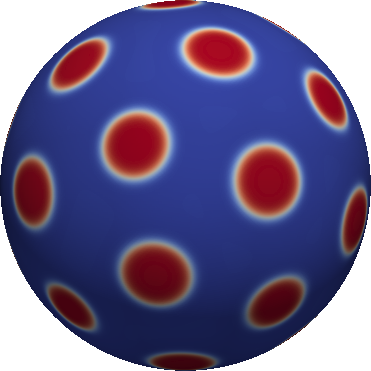}
    }
    \hfill
    \subfigure[$c_2 = 100$]{
      \includegraphics*[width=0.16\textwidth]{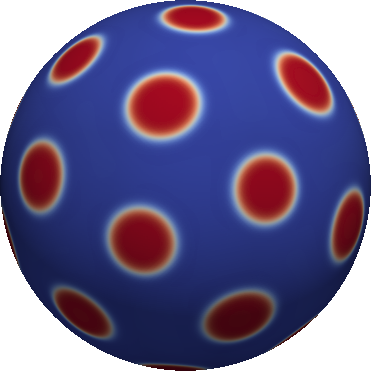}
    }
    \hfill
    \subfigure[$c_2 = 500$]{
      \includegraphics*[width=0.16\textwidth]{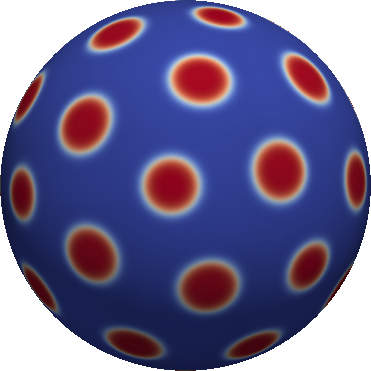}
    }
    \hfill
    \subfigure[$c_2 = 2000$]{
      \includegraphics*[width=0.16\textwidth]{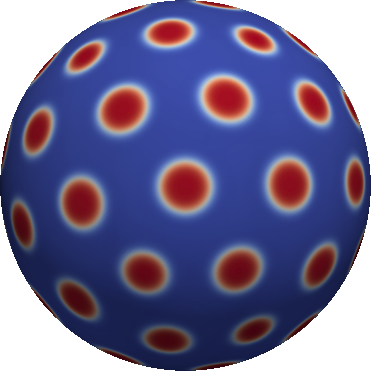}
    }
    \hfill
    \subfigure[]{
      \includegraphics*[height=0.09\textheight]{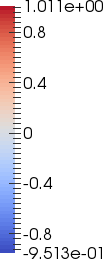}
    }
    \hfill
  }
  \caption{\label{fig:varyC2}\footnotesize Almost stationary discrete
    solutions $\varphi_h$ for different values of $c_2$.} 
\end{figure}

\subsubsection{Comparison with Ohta--Kawasaki model}
\label{subsubsec:OhtaKawasaki}

For the result in Figure \ref{fig:compareToOK} we have investigated the
following scenario. We run a simulation of
\eqref{eq:nonlocal1}--\eqref{eq:uNonlocal} with parameter values as in
\eqref{eq:parametersBasic} except $\delta = 0.0001$ towards an almost
stationary state (see Figure \ref{fig:compareToOK}, left). With the
resulting discrete order parameter $\varphi_h$ we continue a
simulation with Ohta--Kawasaki-based dynamics with parameters
according to \eqref{eq:OKParam1} and \eqref{eq:OKParam2} again towards
an almost stationary state (see Figure \ref{fig:compareToOK},
middle). The difference between these two stationary results for
$\varphi_h$ is displayed in Figure \ref{fig:compareToOK} (right). These
results confirm the analytic results of Section
\ref{subsec:exchangeTerm}. 
\begin{figure}[here]
  \hfill
  \includegraphics*[height=0.12\textheight]{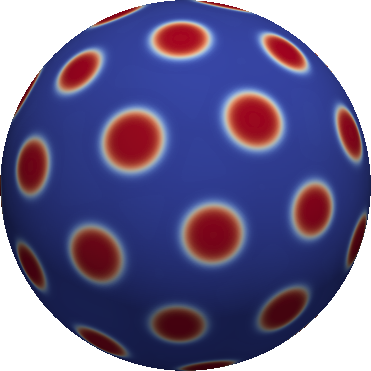}
  \hfill
  \includegraphics*[height=0.12\textheight]{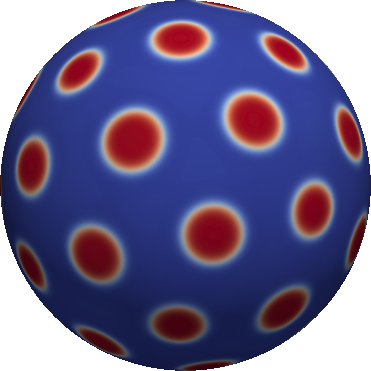}
  \hfill
  \includegraphics*[height=0.12\textheight]{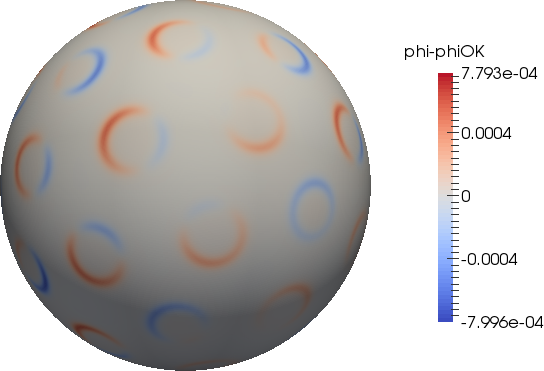}
  \hfill
  \caption{\label{fig:compareToOK}\footnotesize Almost stationary
    $\varphi_h$ obtained from a simulation of the reduced system
    (left) and from subsequent Ohta--Kawasaki-based dynamics (middle),
    difference between the previous numerical solutions (right).}
\end{figure}

\subsubsection{Non--spherical membrane shape}

In a further example, we show in Figure \ref{fig:nonSphere} results
with the standard parameter set  \eqref{eq:parametersBasic} for the case that $\Gamma$ is not a sphere. One obtains similar patterns as for the
sphere. However, this configuration appears to be more stable than the
previous spherical ones, where sometimes very slow arrangements could
take place. This behavior has not been observed for this geometry. We
remark that for this simulation, we have chosen a similar resolution
as for the previous examples with spherical geometry. 

\begin{figure}[here]
  \centerline{
    \hfill
    \subfigure[$\varphi_h(\cdot,t = 0)$]{
      \includegraphics*[width=0.2\textwidth]{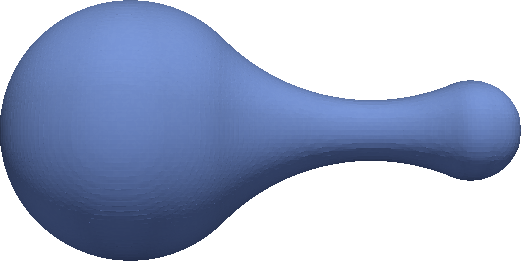}
    }
    \hfill
    \subfigure[$\varphi_h(\cdot,t \approx 0.0053)$]{
      \includegraphics*[width=0.2\textwidth]{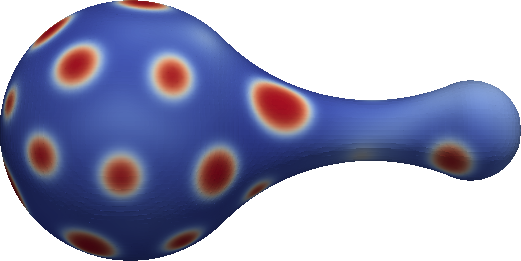}
    }
    \hfill
    \subfigure[$\varphi_h(\cdot,t \approx 1.4133)$]{
      \includegraphics*[width=0.2\textwidth]{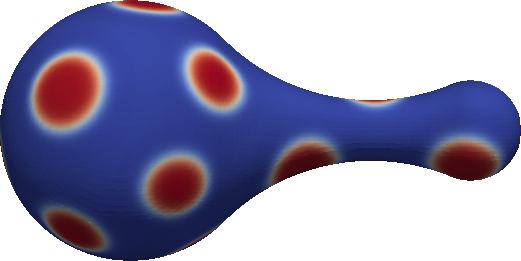}
    }
    \hfill
    \subfigure[$\varphi_h(\cdot,t \approx 55.3218)$]{
      \includegraphics*[width=0.2\textwidth]{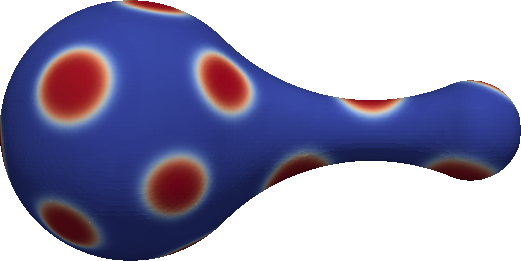}
    }
    \hfill
    \subfigure[]{
      \includegraphics*[height=0.09\textheight]{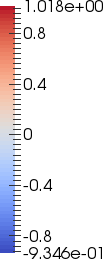}
    }
    \hfill
  }
  \caption{\label{fig:nonSphere}\footnotesize Numerical results for
    a non--spherical surface and $\varphi(\cdot,0) = -0.5 + \mathcal
    R$. Contour plots of $\varphi(\cdot,t)$ for several choices of
    times $t$.}
\end{figure}

\subsubsection{Comparison with energy decreasing dynamics}\label{sec:5.3.4}

Here, we study a choice of $q$ leading to an energy decreasing
evolution. From the analysis in Section
\ref{subsec:energyDecreasingEvolution}, we expect that solutions of
the system \eqref{eq:nonlocal1}--\eqref{eq:uNonlocal} generically converge
to configurations with one lipid phase concentrated in a single geodesic
ball. For this purpose, we choose
\begin{equation*}
  q = -c (\theta - u)
\end{equation*}
with $c=500$, corresponding to \eqref{TCQ} and $f_b(u) = \frac12 u^2$ in
\eqref{eq:mu_u}. The results in this case are displayed in Figure
\ref{fig:thermo} and illustrate the evolution towards an almost
stationary state with a single spherical raft particle. 

\begin{figure}[here]
  \centerline{
    \hfill
    \subfigure[$\varphi_h(\cdot,t = 0)$]{
      \includegraphics*[width=0.16\textwidth]{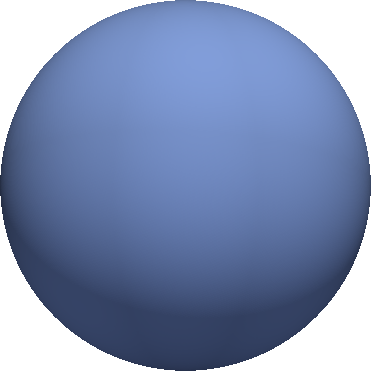}
    }
    \hfill
    \subfigure[$\varphi_h(\cdot,t \approx 0.0030)$]{
      \includegraphics*[width=0.16\textwidth]{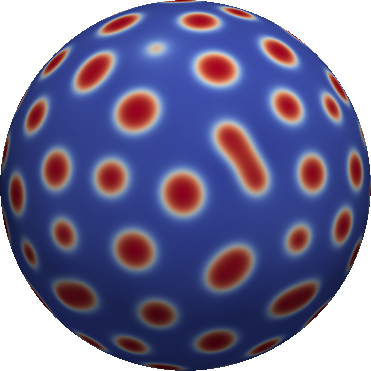}
    }
    \hfill
    \subfigure[$\varphi_h(\cdot,t \approx 0.0103)$]{
      \includegraphics*[width=0.16\textwidth]{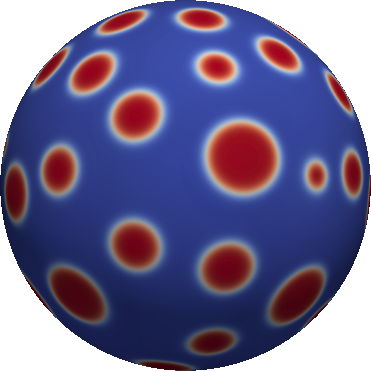}
    }
    \hfill
    \subfigure[$\varphi_h(\cdot,t \approx 1.0166)$]{
      \includegraphics*[width=0.16\textwidth]{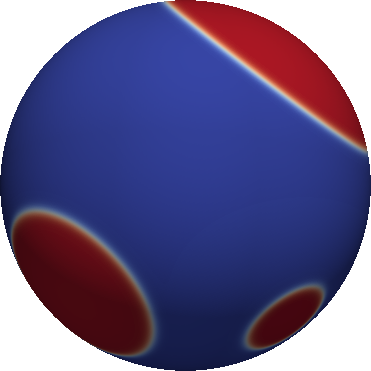}
    }
    \hfill
    \subfigure[$\varphi_h(\cdot,t \approx 2.0562)$]{
      \includegraphics*[width=0.16\textwidth]{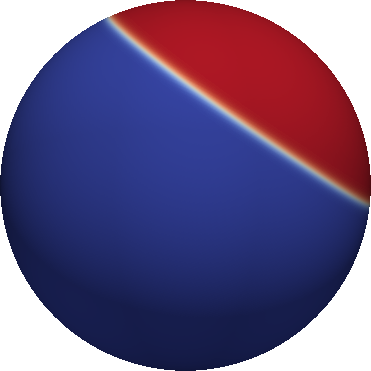}
    }
    \hfill
    \subfigure[]{
      \includegraphics*[height=0.09\textheight]{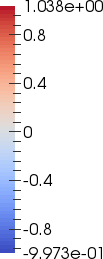}
    }
    \hfill
  }
  \caption{\label{fig:thermo}\footnotesize Numerical results for
    a choice of $q$ leading to an energy decreasing evolution and
    $\varphi(\cdot,0) = -0.5 + \mathcal R$. Contour plots of
    $\varphi(\cdot,t)$ for several choices of times $t$.}
\end{figure}

\subsection{Simulation of the full model}
\label{subsec:diffuse}

For the simulation of the coupled bulk--surface model
\eqref{eq:orig_mu}--\eqref{eq:q-1}, we propose a diffuse interface
approximation based on the diffuse approaches for the treatment of
PDE's on and inside closed
surfaces in \cite{RaVo06} and \cite{LiLoRaVo09}, respectively. For a related phase-field description of a coupled
bulk--surface reaction--diffusion system, we refer to
\cite{RaRo14}. Moreover, a further diffuse interface description of a
coupled bulk--surface system has been given in \cite{TeLiLoWaVo09}, and
in \cite{AbLaSt15} a diffuse interface approach of a linear coupled
elliptic PDE system has been analyzed with respect well-posedness of
the diffuse system and convergence towards its sharp interface
counterpart. 

We choose a domain $\Omega \subset \R^3$ containing $\overline{B}$ and
provide a phase-field approximation of
\eqref{eq:orig_mu}--\eqref{eq:v} by
\begin{alignat}{2}
  \label{eq:diffAppr1}
  \psi \pd_t u &= D \D \cdot (\psi \D u) -\eps_g^{-1}
  b(\psi) q(u,v)&\qquad \text{in } \Omega \times (0,T]&,\\
  \label{eq:diffAppr2}
  b(\psi) \pd_t \varphi &= \D \cdot
  (b(\psi) \D \mu) 
  &\qquad \text{in } \Omega \times (0,T]&,\\
  \label{eq:diffAppr3}
  b(\psi) \mu &=  -\eps \D \cdot
  (b(\psi) \D \varphi) + \eps^{-1}b(\psi)W'(\varphi) 
  - \delta^{-1}b(\psi)(2v-1-\varphi)
  &\qquad \text{in } \Omega \times (0,T]&,\\
  \label{eq:diffAppr4}
  b(\psi) \pd_t v &=  \frac{4}{\delta} \D \cdot
  (b(\psi) \D v)
  - \frac{2}{\delta} \D \cdot (b(\psi) \D \varphi)
  + b(\psi)q(u,v)
  &\qquad \text{in } \Omega \times (0,T]&,
\end{alignat}
where the phase-field function $\psi : \Omega \to \R$ describing the
geometry, that is $B$ and $\Gamma$, is given by 
\begin{equation}
  \label{eq:tanh}
  \psi(x) := \frac12\left(1-\tanh\left(
    \frac{3r(x)}{\eps_g}\right)\right)
\end{equation}
for a (small) real number $\eps_g > 0$ and a signed distance $r$ to
$\Gamma$ being negative in $B$. Note that $\psi$ is obtained by
smearing out the characteristic function $\chi_B$ of $B$ on a length
proportional to $\eps_g$. Moreover, $b = b(\psi) :=
36\psi^2(1-\psi)^2$ is small outside the diffuse interface
region. 

We use a semi-implicit adaptive FEM discretization (see \cite{RaRo14})
and choose $\Omega = (-2,2)^3$ with all solutions assumed
$\Omega$-periodic, $\eps_g = \frac18$, $D=100$. All otherwise
degenerate mobility functions appearing in second order terms in
\eqref{eq:diffAppr1}--\eqref{eq:diffAppr4} are regularized by addition
of a parameter $\delta_r = 10^{-5}$. The resulting linear system is
solved by a stabilized bi-conjugate gradient (BiCGStab) solver. The
numerical scheme is implemented in AMDiS \cite{VeVo07}. Moreover, we
apply the parameters in \eqref{eq:parametersBasic} and initial
conditions as in Section  \ref{subsec:simulation} and additionally
$u(\cdot,0) = 4.25$ corresponding to $M=\frac{20\pi}{3}$ in
\eqref{eq:parametersBasic}. In Figure \ref{fig:diffuse} we see plots
of $\varphi_h(\cdot,t)|_{\{\psi_h=\frac12\}}$ for various times $t$,
where $\psi_h$ is obtained by replacing $r$ by a discrete signed
distance $r_h$ in \eqref{eq:tanh}. The results are in good qualitative
agreement with the results in Section  \ref{subsec:simulation} (see
Figure \ref{fig:basic}) and hence further justify the reduced model
formally obtained in the limit $D \to \infty$.

\begin{figure}[here]
  \centerline{
    \hfill
    \subfigure[$\varphi_h(\cdot,t = 0)$]{
      \includegraphics*[width=0.16\textwidth]{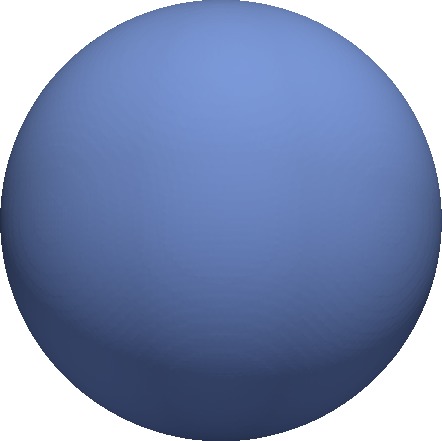}
    }
    \hfill
    \subfigure[$\varphi_h(\cdot,t \approx 0.0070)$]{
      \includegraphics*[width=0.16\textwidth]{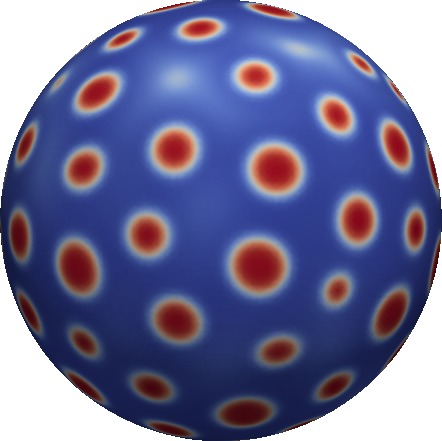}
    }
    \hfill
    \subfigure[$\varphi_h(\cdot,t \approx 0.0126)$]{
      \includegraphics*[width=0.16\textwidth]{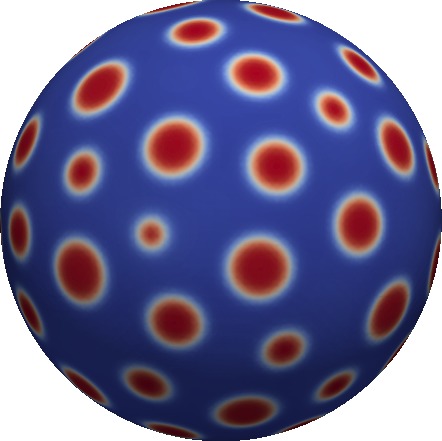}
    }
    \hfill
    \subfigure[$\varphi_h(\cdot,t \approx 0.1356)$]{
      \includegraphics*[width=0.16\textwidth]{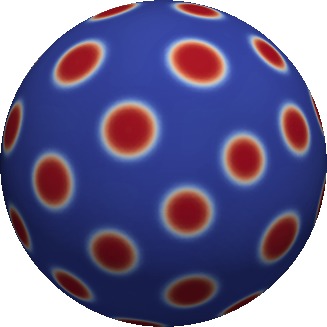}
    }
    \hfill
    \subfigure[$\varphi_h(\cdot,t \approx 2.2172)$]{
      \includegraphics*[width=0.16\textwidth]{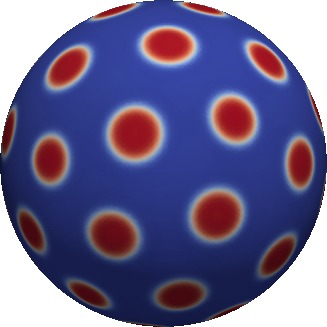}
    }
    \hfill
    \subfigure[]{
      \includegraphics*[height=0.09\textheight]{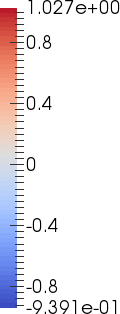}
    }
    \hfill
  }
  \caption{\label{fig:diffuse}\footnotesize Numerical results for
    diffuse interface approximation of the full model
    \eqref{eq:orig_mu}--\eqref{eq:q-1}. Contour plots of
    $\varphi(\cdot,t)|_{\{\psi_h=\frac12\}}$ for several choices of
    times $t$.}
\end{figure}

%==========================================
% conclusions
%==========================================
\section{Conclusions}

We have presented a model for lipid raft formation that extends in particular the model by G{\'o}mez, Sagu{\'e}s and Reigada \cite{GoSR08}. The key new aspect in our work is an explicit account for the cholesterol dynamics in the cytosol, which leads to a complex system of partial differential equations both in the bulk and on the cell membrane. These are coupled by an outflow condition for the cytosolic cholesterol and a source term in the surface membrane equation, determined by a constitutive relation. These considerations lead to an interesting and complex mathematical model. The surface equations combine a Cahn--Hilliard type evolution equation for an order parameter and an equation for the surface cholesterol. The latter is a diffusion equation including a cross-diffusion term and an additional (nonlocal) term. In the bulk we have a diffusion equation with a Robin-type boundary condition that couples both systems.

We have shown that our model can be derived from thermodynamical conservation laws and free energy inequalities for bulk and surface processes. Depending on the specific choice of the exchange term we obtain a total free energy decrease or not, which can be seen as a distinction between equilibrium (closed system) and non-equilibrium (open system) type models. The analysis of both classes reveals striking differences in the behavior: whereas equilibrium-type models only support the formation of macrodomains and connected phases, a prototypical choice of a non-equilibrium model leads to the formation of raft-like structures. 

These findings are the result of both a (formal) qualitative mathematical analysis and careful numerical simulations. More precisely, in a certain parameter regime (large interaction strength between lipid and cholesterol) we have demonstrated that stationary states of the prototypical non-equilibrium model are close to stationary states of the Otha--Kawasaki model, whereas in the case of equilibrium models stationary states coincide with stationary states for a surface Cahn--Hilliard equation. 

For a better understanding of the (complex) model asymptotic reductions are instrumental. We in particular justify the sharp interface limit of the model, that is the reduction when the parameter $\eps$ (related to the width of transition layers between the distinct lipid phases) tends to zero. Moreover, we have derived a simplified model by taking the cytosolic diffusion (typically much larger than the lateral membrane diffusion) to infinity. This reduction leads to a non-local model that only includes variables on the surface membrane.

Numerical simulation show that, depending on the parameter choices, uniformly distributed circular microdomains or stripe pattern that stretch out over the membrane emerge. A justification of the numerical scheme and implementation is presented and the dependence on the various parameters is analyzed. Further evidence for the connection with the Ohta--Kawasaki dynamics (with specific parameters that are in functional relation to the parameters of our model) is given.

Our results support the belief that non-equilibrium processes are indeed essential for microdomain formation. On the other hand, emerging structures are---in the long-time behavior---very regular and uniform, in contrast to the picture of a persisting redistribution of location and sizes of rafts observed in experiments. Thermal fluctuations and/or the interaction with proteins on the membrane have to be taken into account to obtain this more complex behavior. Our contribution presents a solid basis for such extensions.
% ==========================================


\begin{thebibliography}{10}
\bibitem{Maple}
Maple\texttrademark 16, maple is a trademark of waterloo maple inc.
\end{thebibliography}
\end{document}